\documentclass[sn-mathphys-num]{sn-jnl}
\usepackage[utf8]{inputenc}
\usepackage[T1]{fontenc}
\usepackage{graphicx}%
\usepackage{multirow}%
\usepackage{amsmath,amssymb,amsfonts}%
\usepackage{amsthm}%
\usepackage{mathrsfs}%
\usepackage[title]{appendix}%
\usepackage{xcolor}%
\usepackage{textcomp}%
\usepackage{manyfoot}%
\usepackage{booktabs}%
\usepackage{algorithm}%
\usepackage{algorithmicx}%
\usepackage{algpseudocode}%
\usepackage{listings}%
\usepackage{enumerate}
\usepackage{enumitem}
\usepackage{lipsum}
\usepackage{titlesec}
\theoremstyle{thmstyleone}%
\newtheorem{theorem}{Theorem}
\newtheorem{proposition}[theorem]{Proposition}%
\theoremstyle{thmstyletwo}%
\newtheorem{rem}{Remark}%
\newtheorem*{claim*}{Claim}
\newtheorem{lemma}[theorem]{Lemma}
\newtheorem{corollary}[theorem]{Corollary}
\theoremstyle{thmstylethree}%
\newtheorem{definition}{Definition}%
\numberwithin{equation}{section}
\numberwithin{theorem}{section}

\newcommand{\R}{\mathbb{R}}

\newcommand{\N}{\mathbb{N}}

\newcommand{\dist}{\mathrm{dist}}

\newcommand{\IR}{\mathbb{IR}}

\DeclareMathOperator{\diam}{\mathbf{diam}}
\DeclareMathOperator{\midpoint}{\mathbf{mid}}
\DeclareMathOperator{\interior}{\mathbf{int}}
\newcommand{\epsi}{\ensuremath{\varepsilon}}

\newcommand{\authorfootnote}[1]{%
	\begingroup
	\renewcommand\thefootnote{\arabic{footnote}}\footnote{#1}%
	\endgroup
}

\begin{document}

\title[Stable periodic orbits for delay differential equations with unimodal feedback]{Stable periodic orbits for delay differential equations with unimodal feedback}


\author[1]{\fnm{G\'abor} \sur{Benedek}\authorfootnote{G. Benedek has been supported by the M\'oricz Doctoral Scholarship, 2022-23.}}
\email{benedek.gabor.istvan@stud.u-szeged.hu}

\author[2]{\fnm{Tibor} \sur{Krisztin}\authorfootnote{
This research has been supported by the National Research, Development and Innovation Fund, Hungary [project no.  TKP2021-NVA-09, and grants K-129322, KKP-129877] and by the National Laboratory for Health Security [RRF-2.3.1-21-2022-00006].}}
\email{krisztin@math.u-szeged.hu}

\author[3]{\fnm{Robert} \sur{Szczelina}\authorfootnote{R. Szczelina has been supported by 
		Polish national strategic program for the excellence of science
		,,Inicjatywa Doskona{\l}o\'sci - Uniwersytet Jagiello\'nski'' and  
		by National Science Centre, Poland, grant no. 2023/49/B/ST6/02801. 
		For the purpose of Open Access, the author has applied a 
		CC-BY public copyright licence to any Author Accepted Manuscript 
		(AAM) version arising from this submission.}}
\email{robert.szczelina@uj.edu.pl}

\affil[1]{\orgdiv{Bolyai Institute}, \orgname{University of Szeged}, \orgaddress{\street{Aradi v\'ertan\'uk tere 1}, \city{Szeged}, \postcode{H-6720}, \country{Hungary}}}

\affil[2]{\orgdiv{HUN-REN-SZTE Analysis and Applications Research Group, Bolyai Institute}, \orgname{University of Szeged}, \orgaddress{\street{Aradi v\'ertan\'uk tere 1}, \city{Szeged}, \postcode{H-6720}, \country{Hungary}}}

\affil[3]{\orgdiv{Faculty of Mathematics and Computer Science}, \orgname{Jagiellonian University}, \orgaddress{\street{{\L}ojasiewicza 6}, \city{Krak\'ow}, \postcode{30-348}, \country{Poland}}}


\abstract{	We consider delay differential equations of the form $ y'(t)=-ay(t)+bf(y(t-1)) $ with positive parameters $a,b$ and a unimodal  $f:[0,\infty)\to [0,1]$. 
It is assumed that the nonlinear $f$ is close to a function $g:[0,\infty)\to [0,1]$  with $g(\xi)=0$ for all $\xi>1$.
The fact  $g(\xi)=0$ for all $\xi>1$ allows to construct stable periodic orbits for the equation $x'(t)=-cx(t)+dg(x(t-1))$ with some parameters $d>c>0$. 
Then it is shown that the equation $ y'(t)=-ay(t)+bf(y(t-1)) $ also 
has a stable periodic orbit provided $a,b,f$ are sufficiently close to $c,d,g$ in a certain sense. 
The examples include  $f(\xi)=\frac{\xi^k}{1+\xi^n}$ for parameters $k>0$ and  $n>0$ together with the discontinuous $g(\xi)=\xi^k$ for $\xi\in[0,1)$, and $g(\xi)=0$ for $\xi>1$. 
The case $k=1$ is the famous  Mackey--Glass equation, 
the case $k>1$ appears in population models with Allee effect, and the case  $k\in(0,1)$ arises in some economic growth models.  
The obtained stable periodic orbits may have complicated structures.}





\maketitle

\setcounter{footnote}{0}

\section{Introduction}\label{sec1}
Delay differential equations of the form
\begin{equation}
	y'(t)=-ay(t)+bh(y(t-\sigma))
	\label{eqn:h}
\end{equation} 
with positive parameters $a,b,\sigma$ and a nonlinear function $h:[0,\infty)\to [0,\infty)$ 
arise in a wide range of applications. By rescaling the time we may assume $\sigma=1$. 

For monotone nonlinearities $h$ the dynamics is relatively simple: there is a Morse decomposition of the global attractor, on each Morse component the dynamics is planar, and a fine structure of the global attractor can be described (\cite{K1}, \cite{K5}, \cite{KWW}, \cite{MPS}). 
A non-monotone, in particular a  unimodal nonlinearity $h$ can cause entirely different dynamics compared to the monotone case (\cite{W0}, \cite{W1}, \cite{W2}, \cite{W3}). Function $h$ is called unimodal if it has exactly one extremum, and changes its monotonicity at only one point. 
In 1977 Lasota \cite{LW}, and Mackey and Glass \cite{MC} independently proposed Equation \eqref{eqn:h} with a unimodal nonlinearity $h$ to model the feedback control of blood cells. 
Lasota's nonlinearity was
$f(\xi)= \xi^k e^{-\xi},$ 
while Mackey and Glass introduced $f(\xi)=\frac{\xi}{1+\xi^n}$ 
for some positive parameters $k, n$. 
The work of Mackey and Glass \cite{MC} stimulated an intensive research on equations of the form \eqref{eqn:h} with non-monotone nonlinearities. 
See \cite{HOW} for a review of the impact of the Mackey--Glass equation on 
nonlinear dynamics. 

There are several results on the dynamics generated by Equation \eqref{eqn:h}. 
Some papers consider the delay $\sigma$ as a parameter in Equation \eqref{eqn:h}, and the coefficients $a,b$ may depend on the delay as well. 
The theoretical results show stability of equilibria, existence of oscillations and periodic behavior, local and global Hopf bifurcations, establish some connecting orbits, see e.g.,
\cite{BBI}, \cite{Kuang}, \cite{Smith}, \cite{Ruan},\cite{Wei}. 
The papers   
\cite{FCP}, \cite{LizRost1}, \cite{LizRost2}, \cite{RostW} 
estimate the size of the global attractor.
Bistable nonlinearities are studied in 
\cite{HYZ}, \cite{HWW}, \cite{Lin}, \cite{LizRH}. 
The numerical results \cite{GS}, \cite{DH}, \cite{Morozov} demonstrate the complexity of the dynamics.   
Despite the large number of rigorous mathematical results, numerical and experimental studies, the dynamics of Equation \eqref{eqn:h} is not completely understood yet. 
There is still no rigorous proof of chaos for nonlinearities like 
$\xi e^{-\xi}$, $ \frac{\xi}{1+\xi^n} $, $ \frac{\xi^2}{1+\xi^n} $. 
However, \cite{Lani-W} shows chaos for some constructed unimodal nonlinearities. 

In this paper we consider a pair of delay differential equations 
\begin{equation*}\label{eqn:Ef}
	\tag{$E_{f}$}
	y'(t)=-ay(t)+bf(y(t-1))
\end{equation*} 
and 
\begin{equation}\label{eqn:Eg}
	\tag{$E_{g}$}
	x'(t)=-cx(t)+dg(x(t-1))
\end{equation}
with positive parameters $a,b,c,d$ and nonlinearities 
$f,g$ satisfying 
\begin{enumerate}
	\item[(F)]  $f:[0,\infty)\to [0,1]$ is continuous, $f|_{(0,\infty)}$ is $C^1$-smooth,  $f(0)=0$, and
	$ f(\xi)<1$ for all $\xi>1$,  
\end{enumerate}
and 
\begin{enumerate}
	\item[(G)]   $g:[0,\infty)\to [0,1] $,   $g|_{(0,1)}$ is $C^1$-smooth, $g'>0$ on $(0,1)$, $g(0)=0=\lim_{\xi\to 0+}g(\xi)$, $\limsup_{\xi\to 1-}g'(\xi)<\infty$, 
	and $g(\xi)=0$ for all $\xi>1$. 
\end{enumerate}

Function $g$ is an extreme case of a unimodal feedback, see Figure \eqref{fig:figf}.
Equation  \eqref{eqn:Eg} in itself is an interesting and nontrivial model equation for delayed feedback. 
Equation  \eqref{eqn:Eg} looks simple because of $g(\xi)=0$ for $\xi>1$. 
For $t$-intervals where $x(t-1)>1$, Equation \eqref{eqn:Eg} reduces to $x'(t)=-cx(t)$.  
This fact allows to construct certain special solutions of \eqref{eqn:Eg}, for example 
periodic solutions for some $d>c>0$. 
The idea is that if the parameters $a,b$ are close to $c,d$, respectively, and function 
$f$ is close to $g$ in a certain sense given below, then  
the special solutions obtained for Equation \eqref{eqn:Eg} are 
preserved for Equation \eqref{eqn:Ef}. 
In this manuscript the idea is worked out for periodic solutions: 
we construct stable priodic orbits for  Equation \eqref{eqn:Eg}, and 
prove that Equation \eqref{eqn:Ef} has stable periodic orbits as well, provided $a,b,f$ are close to $c,d,g$.

Examples include the nonlinearities 
\begin{equation}\label{proto}
	f(\xi)=\dfrac{\xi^k}{1+\xi^n}
\end{equation}
and 
\begin{equation}\label{proto2}
	f(\xi)=\xi^k e^{-\xi^n}
\end{equation}
with parameters $k>0$ and $n>0$. 
In these cases function $g$ is the 
pointwise limit of $\frac{\xi^k}{1+\xi^n}$ and  $ \xi^k e^{-\xi^n} $
for $\xi\in [0,\infty)\setminus\{1\}$, as $n\to\infty$, 
that is, 
$g(\xi)=\xi^k$ if $0\leq \xi < 1$, and $g(\xi)=0$ for $\xi>1$. See Figure 1. 
The value of $g(1)$ is irrelevant. 
The result, for the examples is that if Equation \eqref{eqn:Eg} has a stable periodic orbit for some $c,d$, and  $a,b$ are close to $c,d$, and $n$ is sufficiently large, then   \eqref{eqn:Ef} has a stable periodic orbit. 

Example \eqref{proto} with $k=1$ is the famous Mackey--Glass nonlinearity \cite{MC}. 
The case $k>1$ appears in models of population dynamics with Allee effect \cite{Morozov}, \cite{HWW}.
In particular, for $k=2$, the paper \cite{Morozov} numerically observed 
a variety of dynamical behavior depending on the parameters. 
The case $k\in (0,1)$ arises in a variant of a neoclassical economic growth model by replacing the so called pollution function 
$e^{-\delta \xi}$ by $\frac{1}{1+\xi^n}$ or by $ e^{-\xi^n} $, see \cite{Matsumoto}. 

\begin{figure}[h]
	\centering
	\includegraphics[width=1\linewidth]{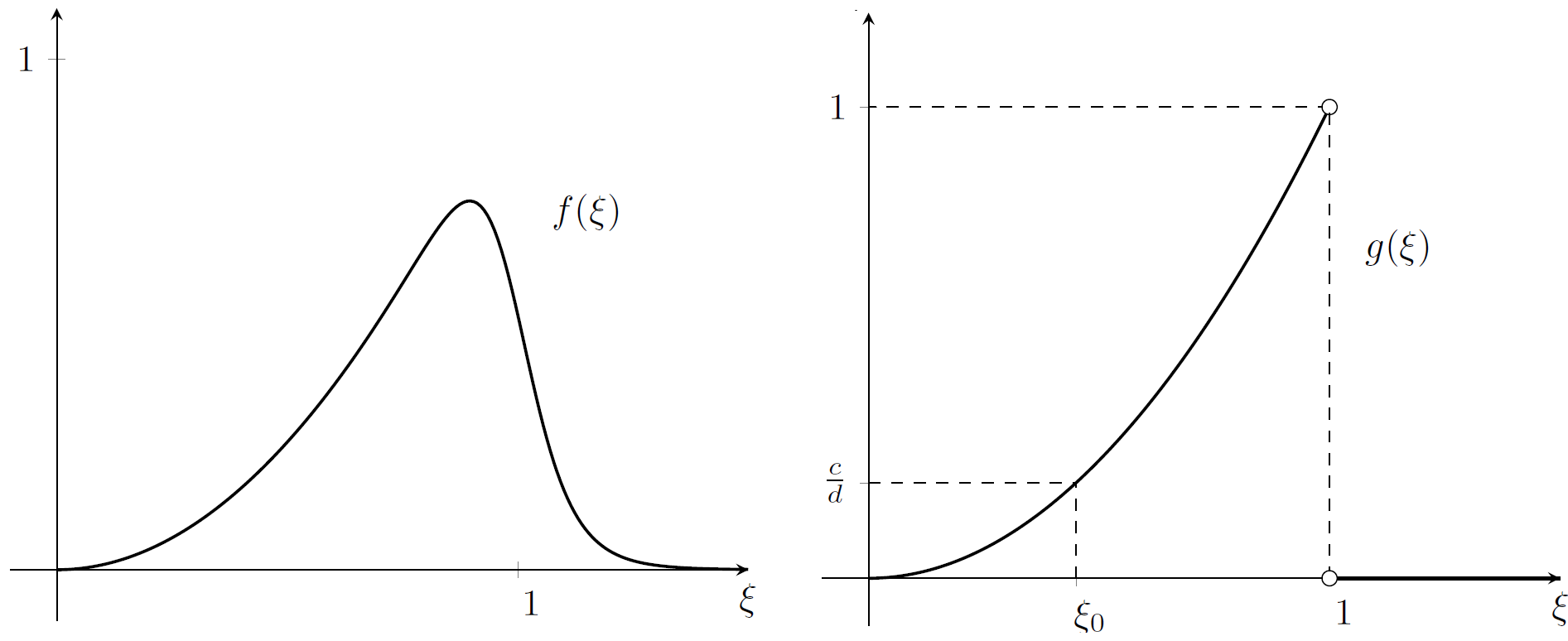}
	\caption{The shape of functions $f(\xi)=\frac{\xi^k}{1+\xi^n}$ (left) and $g(\xi)=\xi^k$ (right) satisfying conditions (F) and (G) respectively in case $k=2$.}
	\label{fig:figf}
\end{figure}

The idea of this paper originates from \cite{BKV} where $f(\xi)=\frac{\xi}{1+\xi^n}$ was the Mackey--Glass nonlinearity with $g(\xi)=\xi$ for $\xi\in[0,1)$, $g(\xi)=0$ for $\xi>1$.  
In \cite{BKV} the linearity  of $g$ in $[0,1)$ was essential to construct 
complicated looking periodic solutions of \eqref{eqn:Eg}. 

Here we extend \cite{BKV} to all nonlinearities $f$ close to $g$ satisfying 
only condition (G). This includes not only the prototype nonlinearities 
\eqref{proto} and \eqref{proto2} for large $n$, but also a large class of other unimodal nonlinearities. In addition, for the  prototype nonlinearities 
\eqref{proto} and \eqref{proto2}, a threshold number  can be explicitly constructed so that the results are valid if $n$ is larger than the threshold number. 

Most of the periodic orbits obtained here do not seem to come from local or global  Hopf bifurcations.  Projecting the periodic solution $q$ into the plane by $t\mapsto 
(q(t),q(t-1))\in\R^2$ produce complicated looking figures for some periodic solutions $q$, although the corresponding periodic orbits are stable.  

The region of attraction of the stable periodic orbit of \eqref{eqn:Ef} is estimated as well. Therefore we rigorously show bistability for the prototype  nonlinearities if $k>1$. 
One stable attractor is the trivial equilibrium $0$, the other one is the 
obtained stable periodic orbit.

We expect that this paper with \cite{BKV} opens up a new direction to study a wide class of delay differential equations with nonlinearities close to 
$g$. In particular, in forthcoming papers,  motivated by the numerical results of \cite{Morozov}, we consider the nonlinearity $f(\xi)=\frac{\xi^2}{1+\xi^n}$ 
modeling Allee effect, prove connections from equilibria to periodic orbits and between periodic orbits, moreover, we show the existence of homoclinic orbits. By the approach of the present paper and \cite{BKV} we hope to prove  rigorously complicated dynamics for the prototype nonlinearities for large values of $n$. 
Remark that the paper \cite{I} constructed homoclinic orbits for certain unimodal nonlinearities, however not including \eqref{proto}, \eqref{proto2}, and not having the spectral condition expected to be necessary for Shilnikov chaos.  

Now we give the main results of the paper by using more technicalities. 
Let $g^{-1}$ denote the inverse of the restriction $g|_{[0,1)}$.
The following property of Equation \eqref{eqn:Eg} plays an important role.  
\begin{enumerate}
	\item[(P)] The parameters  $d>c>0$ and the function $g$ are given such that 
	(G) holds and  Equation \eqref{eqn:Eg} has a periodic solution $p:\mathbb{R}\to (0,\infty)$ with minimal period $\omega_p>0$, and  
	\begin{enumerate}
		\item[(P1)] $p(0)=1 $, $p(t)>1$ for all $t\in [-1,0)$; 
		\item[(P2)] $(p(t),p(t-1))\ne \left(1,g^{-1}\left(\frac{c}{d}\right)\right)$ for all $t\in[0,\omega_p]$. 
	\end{enumerate}
\end{enumerate}
Set $\mathcal{O}_p=\{ {p}_t:\, t\in [0,\omega_p]\}$, the orbit of $p$. 

In order to define the closeness of $f$ to the discontinuous $g$, for a function $u: I\to\mathbb{R}$, given on an interval $I$, introduce   
$$\|u\|_I=\sup_{\xi\in I}|u(\xi)|.$$

The main result can be formulated as follows.

\begin{theorem}\label{main_new}
	Suppose that property (P) holds. 
	Choose $\kappa_1>0$, $\kappa_2>0$ such that 
	$p(\mathbb{R})\subset (\kappa_1,\kappa_2)$, and let $m=\left[ \omega_p\right] $ be the integer part of $\omega_p$. 
	There exists an $\varepsilon>0$ so that if 
	$a, b, f$ are close to $c,d,g$, respectively, in the sense that
	\begin{itemize}
		\item[(1)]	
		$|a-c|<\varepsilon,\quad |b-d|<\varepsilon,$
		\item[(2)]	$\|f-g\|_{[\kappa_1,1-\varepsilon]\cup[1+\varepsilon,\kappa_2]}<\varepsilon$,
		\item[(3)]	$\|f'\|_{[1+\varepsilon,\kappa_2]}\left( \|f'\|_{[\kappa_1,\kappa_2]}\right)^m
		<\varepsilon$, 
	\end{itemize}
	then equation \eqref{eqn:Ef} has a periodic solution $q:\mathbb{R}\to (\kappa_1,\kappa_2)$, with minimal period $\omega_q$ so that  $|\omega_p-\omega_q|<\varepsilon$,   and
	$
	{\mathcal{O}_q}=\{ {q}_t:\, t\in [0,\omega_q]\}
	$ 
	is hyperbolic, orbitally stable, exponentially attractive with asymptotic phase,  and  $\dist\,({\mathcal{O}_p},\mathcal{O}_q)<\varepsilon.$ 
	In addition, the set 
	$$
	\{\psi \in C^+:\quad \psi(s)\in [1+\varepsilon,\kappa_2] \text{ for all } s\in [-1,0]\}
	$$ 
	belongs to the region of attraction of $\mathcal{O}_q$. 
	
\end{theorem}

In Section 3, first a more technical version of Theorem \ref{main_new} 
will be shown, see Theorem \ref{main0}. 
The constants in Theorem \ref{main0} expressing the closeness of $a,b,f$ to 
$c,d,g$ are constructed explicitly in terms of the properties of the 
periodic solution $p$ given in (P), see Sections 2--3. 

Condition (3) of Theorem \ref{main_new} is satisfied if the derivative $f'$ of $f$ is small on $[1+\varepsilon,\kappa_2]$ compared to $f'$ on 
$[\kappa_1,\kappa_2]$. This together with condition (2) requires  that 
$f$ is close to $g$ in the $C^1$-norm on $[1+\varepsilon,\kappa_2]$. 
In particular, for $ f_n(\xi)=\frac{\xi^k}{1+\xi^n }$ with a fixed $k>0$, 
$|f_n'(\xi)|\le K_0 n$ on $ [\kappa_1,\kappa_2] $, and $|f_n'(\xi)|\le \frac{K_0n}{1+(1+\varepsilon)^n}$ on $[1+\varepsilon,\kappa_2]$, for some $K_0$, see Section \ref{subsec2}. 
Therefore, condition (3) holds for large $n$. 
Our result does not work for the nonlinearity $f(\xi)=\xi e^{-\xi}$ because of the lack of condition (3).  

Theorems \ref{main_new} and \ref{main0} can guarantee the existence of a stable periodic orbit only if property (P) holds. 
In Section \ref{sec4} an analytic proof shows that (P) is always satisfied if $d$ is sufficiently large compared to $c$, namely
\begin{equation}\label{d>}
	d>\max\left\{ \dfrac{e^{2c}-1}{I_{cg}},\dfrac{c}{g(e^{-c})}\right\}
\end{equation}
where $I_{cg}=\int_1^2 e^{cs}g(e^{-c(s-1)})\,ds $. 
By Lemma \ref{lemma1}  the existence of a periodic solution with (P1) and $d>\frac{c}{g(e^{-c})}$ 
combined imply $p(t)\ge e^{-c}$ for all $t\in\R$. 
Hence
\begin{align*}
	\xi\in [e^{-c},1) &\ \text{ implies }\ -c+dg(\xi)\ge -c+dg(e^{-c})>0,\\
	\xi>1 &\ \text{ implies }\ -c+dg(\xi)= -c <0, 
\end{align*}
that is, a negative feedback condition holds with respect to $\xi=1$, 
see \cite{DVLGW}. 
In this case property (P1) of $p$ and slow oscillation combined 
yield that three consecutive zeros of $p(t)-1$ determine $\omega_p$, 
and (P2) holds as well. 
Consequently, under condition \eqref{d>}, Theorems \ref{main_new} and \ref{main0} give the existence of a stable slowly oscillating periodic 
solution $q$ of Equation \eqref{eqn:Ef}, and the projected curves 
$t\mapsto (p(t),p(t-1))$ are relatively simple. 

The rigorous, computer-assisted verification of property (P) in Section \ref{CAP} 
is an essential part of this paper. 
Indeed, Section \ref{CAP} developes a rigorous numerical technique 
to verify property (P) for some given $c,d,g$ in Equation \eqref{eqn:Eg}. 
For several examples where  (P) is satisfied, 
 the obtained periodic solution $p$ can be complicated in the sense 
that the projected curves 
$t\mapsto (p(t),p(t-1))$ look complicated, similar to the numerically observed chaotic situations obtained in earlier papers. Notice that the complexity of $p$ implies that of $q$.
We emphasize that the analogous computer-assisted proof in \cite{BKV} 
was much simpler, since the fact $g(\xi)=\xi$, $\xi\in[0,1)$, allowed  
the search for periodic solutions, in local coordinates, in   
the class of functions $s\mapsto \sum_{j=0}^N a_j s^j e^{-cs}$. 
The particular form of $g$ makes it possible to verify property (P) for 
some parameters $d>c>0$. 
Property (P) is the same as hypothesis (H) in \cite{BKV} for the case 
$g(\xi)=\xi$, $\xi\in[0,1)$, and $g(\xi)=0$ for $\xi>1$. 
It is nontrivial to extend the verification of (P) from the piece-wise linear $g$ 
of \cite{BKV} to a general $g$ satisfying (G). The method
of Section \ref{CAP} achives this for a wide class of functions $g$,
such that they can be represented on $(0, 1)$ as a composition 
of simple functions such as arithmetic operators and well know  
functions, such as $\exp$, $\sin$, $\cos$, $\sqrt{\cdot}$, etc. 

Property (P1) implies  $p(t)=e^{-ct}$ for $t\in [0,1]$. The values of 
$p$ on $[0,1]$ determine uniquely $p$ on $[1,\infty)$. Therefore, 
$p$ is unique if it exists. From a solution $x:[0,\infty)\to\R$ of Equation \eqref{eqn:Eg} with 
$x(t)=e^{-ct}$ for $t\in [0,1]$ one easily finds a periodic $p$ with property (P1) if $x(t)>1$ for all $t\in [T,T+1]$ for some $T>1$, see Remark \ref{remark}. 
The meaning of (P2) is that the orbit passes through the hyperplane 
$ \{\phi\in C([-1,0],\R): \phi(0)=1\} $
transversally. 
Property (P2) guarantees that the measure of those times $t\in [0,\omega_p]$ is small, where $p(t)$ is close to the discontinuity point $\xi=1$ of $g$. 
This fact is crucial in the estimation of  the difference between the periodic solution 
of Equation \eqref{eqn:Eg} and a solution of Equation \eqref{eqn:Ef}, 
see Proposition \ref{prop:x-y}. 
We remark that the idea of considering a simple nonlinear feedback function  
and nearby nonlinearities was applied in earlier works in the case when 
the nonlinearity $g$ was piece-wise constant,  see e.g. \cite{BVG}, \cite{UDH}, \cite{KPV}, \cite{K3}, \cite{K4}, \cite{OW}, \cite{SW}, \cite{VAS}, \cite{Wa1}, \cite{Wa2}. 
In the piece-wise constant situation the search for periodic solutions can be reduced to a finite dimensional problem. 
Our case is different since $g$ is not a step function on $[0,1)$,    and  the search for periodic solutions is still an infinite dimensional problem. 
Mackey et al. \cite{MOPW} applied a limiting technique, similar 
to the work and \cite{BKV}, for an equation including the Mackey--Glass nonlinearity $ \frac{\xi}{1+\xi^n} $, and obtained simple looking stable periodic solutions. 

\section{Preliminary results}\label{sec2}
In this section we assume that $b>a>0$, $d>c>0$, and conditions (F) and (G) hold. 

We use the notation $\R$ for the set of all real numbers, $\N$ for the set $\{1,2,\dots\}$, and let $\N_0=\{0\}\cup \N$.

Let $C$ denote the Banach space $C([-1,0],\R)$ with the norm $\|\phi\|=\max_{s\in[-1,0]}|\phi(s)|$. For an interval $I$, $t-1,t\in I$ and a continuous function $u:I\to\R$, 
$u_t\in C$ is given by $u_t(s)=u(t+s)$, $-1\le s\le 0$. 
Introduce the subsets of $C$:
$$
C^+=\{\psi\in C: \psi(s)\ge 0, \ s\in [-1,0]\},\quad 
C^+_*=\{\psi\in C^+: \psi^{-1}(c) \text{ is finite for }c\in (0,1]\}
$$
where $\psi^{-1}(c)=\{s\in[-1,0]: \psi(s)=c\}$. 

A solution of Equation \eqref{eqn:Ef} on $[-1,\infty)$ with initial function $\psi\in C^+$ 
is a continuous function $y=y^\psi:[-1,\infty)\to \mathbb{R} $ so that $y_0=\psi$, the restriction $y|_{(0,\infty)}$ is differentiable, and equation  \eqref{eqn:Ef} holds for all $t>0$. 
The solutions are easily obtained from the variation-of-constants formula for ordinary differential equations on successive intervals of length one, that is
\begin{equation}
	y(t)=e^{-a(t-j)}y(j) +b\int_j^t e^{-a(t-s)}f(y(s-1))\,ds, 
\end{equation}
where $j\in \mathbb{N}_0$, $j\le t\le j+1$. Each $\psi\in C^+$ as an initial function uniquely determines a solution 
$y=y^{\psi}:[-1,\infty)\to \mathbb{R}$ with $y_0^{\psi}=\psi$, and 
$y^{\psi}_t\in C^+$ for all $t\ge 0$. 
It is easy to see that $\psi\in C^+$, $\psi(0)>0$ imply $y^\psi(t)>0$ for all $t>0$. 
The solutions define the continuous semiflow 
$$\Phi:[0,\infty)\times C^+\ni (t,\psi)\mapsto y_t^{\psi}\in C^+.$$

The discontinuity of $g$ requires a slightly different definition of a solution for 
\eqref{eqn:Eg}. A solution  of Equation \eqref{eqn:Eg} with an initial function 
$\phi\in C^+$ is a continuous function $x=x^\phi:[-1,t_\phi)\to\mathbb{R} $ with some $0<t_\phi\le\infty$ such that 
$x_0=\phi$, the map $[0,t_\phi)\ni s\mapsto g(x(s-1))\in\mathbb{R}$ is locally integrable, and 
\begin{equation}
	x(t)=e^{-c(t-j)}x(j) +d\int_j^t e^{-c(t-s)}g(x(s-1))\,ds 
\end{equation}
holds for all $j\in \mathbb{N}_0$ and $t\in [0,t_\phi)$ with $j\le t\le j+1$.
It is not difficult to prove that,  
for any $\phi\in C^+$, there is a unique solution $x^\phi$ of Equation \eqref{eqn:Eg} on $[-1,\infty)$. 
For $u\in (-1,1)$ and for the inititial function $\phi(s)=1+u$,  the solution $x^u$ of Equation \eqref{eqn:Eg} is 
$$ 
x^u(t)=e^{-ct}(1+u) +\dfrac{d}{c}g(1+u)[1-e^{-c}]
\qquad (t\in [0,1],
$$
see \cite{BKV}. If continuity holds on initial functions in $C^+$, 
then
$$
\lim_{u\to 0-}x^u(t) = x^0 (t)=\lim_{u\to 0+}x^u(t) 
\qquad \text{for all } t\in [0,1].  
$$
This is impossible since, by (G),  
$\lim_{u\to 0-}g(1+u)>0$, $\lim_{u\to 0+}g(1+u)=0$. 
Consequently, in $C^+$ there is no continuous dependence on initial functions 
for  Equation \eqref{eqn:Eg}. 
This fact motivates the introduction of $C^+_*$ as a phase space for equation \eqref{eqn:Eg}.

The proof of the following result requires only a straightforward modification  of the 
proof of Proposition 2.2 in \cite{BKV}, 
therefore it is omitted here. 

\begin{proposition}\label{prop:exist}
	For all $\phi\in C^+_*$, Equation \eqref{eqn:Eg} has a unique solution $x^\phi$ on $[-1,\infty)$ with  $x_t^\phi\in C^+_*$ for all $t\ge 0$. If $t>0$ and $x^\phi(t-1)\ne 1$, then $x^\phi$ is differentiable at $t$, and equation \eqref{eqn:Eg} holds at $t$. 
	The map $\Gamma:[0,\infty)\times C^+_*\ni (t,\phi)\mapsto x_t^\phi \in C^+_*$ is a continuous semiflow. 
\end{proposition}

We remark that, for $\psi\in C^+$ and $\phi\in C^+_*$, the unique solutions $y^{\psi}$ and 
$x^{\phi}$  of \eqref{eqn:Ef} and \eqref{eqn:Eg} satisfy the integral equations  
\begin{equation}\label{inteq-en}
	y^{\psi}(t)=e^{-a(t-\tau)}y^{\psi}(\tau) +b\int_{\tau}^t e^{-a(t-s)}f(y^{\psi}(s-1))\,ds  
	\qquad (0\le \tau< t<\infty)
\end{equation}
and
\begin{equation}\label{inteq-ei}
	x^\phi(t)=e^{-c(t-\tau)}x^\phi(\tau) +d\int_{\tau}^t e^{-c(t-s)}g(x^\phi(s-1))\,ds 
	\qquad  (0\le \tau< t<\infty),
\end{equation}
respectively.

The following boundedness result is analogous to Proposition 2.3 in  \cite{BKV}, and its proof is omitted:

\begin{proposition}\label{prop-bound}
	Let $\phi\in C_*^+$, $\psi \in C^+$, and let $x=x^\phi$, $y=y^\psi$ be the solutions of Equations \eqref{eqn:Eg} and \eqref{eqn:Ef}, respectively. 
	\begin{itemize}
		\item[(i)]  There exists $t_*(\phi)\ge 0$ so that $x(t)<d/c$ for all $t>t_*(\phi)$.  If $\phi(0)<d/c$ then $x(t)<d/c$ for all $t\geq 0$.
		\item[(ii)] There exists $t_{**}(\psi)\geq 0$ so that $y(t)<b/a$ for all $t>t_{**}(\psi)$. If  $\psi(0)<b/a$ then $y(t)<b/a$ for all $t\geq 0$.
	\end{itemize}
\end{proposition}

Let $g^{-1}$ denote the inverse of $g|_{[0,1)}$. 
By (G), $g^{-1}$ is strictly increasing, continuous, $g^{-1}(0)=0$. Define 
$$\xi_0=g^{-1}\left(\frac{c}{d}\right)\in (0,1).$$ 

\begin{proposition}\label{prop:mu}
	For each $\delta_0\in(0,\min\{\xi_0,1-\xi_0\})$ there exist 
	$\delta_1\in (0,\delta_0)$ and  $\mu>0$
	such that 
	$$
	|-cu+dg(v)|\ge \mu \text{ for all }
	(u,v)\in \left[ 1-\delta_1,1+\delta_1\right] \times \left(\left[0,\xi_0-\delta_0\right]\cup \left[\xi_0+\delta_0,\infty \right)  \right).
	$$
\end{proposition} 

\begin{proof} From condition (G), $d>c>0$, and the choice of $\xi_0$, 
	there exists a $\mu>0$ so that 
	$|-c+dg(v)|\ge 2\mu $ for all $v\in \left[0,\xi_0-\delta_0\right]\cup \left[\xi_0+\delta_0,\infty \right)$. 
	By continuity, it is easy to see that the stated inequality holds if $\delta_1$ is sufficiently small.  
\end{proof}

\begin{proposition}\label{prop:y-z} 
	Let $\delta>0$,   $\ell\in \N_0$, 
	$0<\nu_1<1$, $1+\delta<\nu_2$, and $\psi, \chi \in C^+$ 
	be given such that for the solutions  $y=y^\psi$ and $z=z^\chi$  of \eqref{eqn:Ef} the relations
	$$
	\psi(s)\in [1+\delta,\nu_2],\ \chi(s)\in [1+\delta,\nu_2] \text{ for all }s\in [-1,0], 
	$$
	$$
	y(t)\in [\nu_1,\nu_2], \ z(t)\in [\nu_1,\nu_2] \text{ for all }t\in [0,\ell]
	$$ 
	hold. 
	Then,  for all $t\in[0,\ell+1]$, 
	\begin{equation}\label{ineq:y-z}
		|y(t)-z(t)|\leq \left(|\psi(0)-\chi(0)|+b\|f'\|_{[1+\delta,\nu_2]}\|\psi-\chi\|\right)\left[1+b\|f'\|_{[\nu_1,\nu_2]}\right]^\ell.
	\end{equation} 
\end{proposition}

\begin{proof} 
	From the integral equation \eqref{inteq-en} for $y=y^\psi$ and $z=z^\chi$ 
	with $\tau=0$ and $t\in [0,1]$   we have 	
	\begin{align*}
		\left|y(t)-z(t)\right|& \le e^{-at}|\psi(0)-\chi(0)| + b\sup_{1+\delta\le\xi\le\nu_2}|f'(\xi)|
		\int_0^t e^{-a(t-s)}|y(s-1)-z(s-1)|\,ds\\
		& \leq |\psi(0)-\chi(0)| + b\|f'\|_{[1+\delta,\nu_2]}
		\left\| \psi-\chi\right\|,
	\end{align*}
	that is, \eqref{ineq:y-z} holds for $\ell=0$.  
	Suppose that $k\in\N$  and for all $t\in [0,k]$ the 
	inequality
	$$|y(t)-z(t)|\le \left(\left|\psi(0)-\chi(0)\right|+b\|f'\|_{[1+\delta,\nu_2]}\left\|\psi-\chi \right\| \right)\left[1+b\|f'\|_{[\nu_1,\nu_2]}\right]^{k-1}$$ 
	is valid. 
	Then 
	\begin{equation}\label{ineqk:y-z}
		\|y_k-z_k\|\le \left(\left|\psi(0)-\chi(0)\right|+b\|f'\|_{[1+\delta,\nu_2]}\left\|\psi-\chi \right\| \right)
		\left[1+b\|f'\|_{[\nu_1,\nu_2]}\right]^{k-1}
	\end{equation}  
	holds as well. 	
	By using \eqref{ineqk:y-z}, from the integral equation \eqref{inteq-en}
	for $y$ and $z$ with $\tau=k$ and $t\in [k,k+1]$, we obtain   
	\begin{align*}
		|y(t)-z(t)|\le 
		& e^{-a(t-k)}|y(k)-z(k)| + b\sup_{\xi\in[\nu_1,\nu_2]}|f'(\xi)|
		\int_k^t e^{-a(t-s)} |y(s-1)-z(s-1)|\,ds\\
		\le 
		& \left[1+b\|f'\|_{[\nu_1,\nu_2]}\right] \|y_k-z_k\| \\
		\le 
		& \left(\left|\psi(0)-\chi(0)\right|+b\|f'\|_{[1+\delta,\nu_2]}\left\|\psi-\chi \right\| \right)\left[1+b\|f'\|_{[\nu_1,\nu_2]}\right]^{k}.
	\end{align*} 
	Thus, by the principle of induction, inequality \eqref{ineqk:y-z} is satisfied for all 
	$k\in\{1,\ldots,\ell +1\}$, and the proof is complete. 
\end{proof}

Suppose that the following property holds for Equation \eqref{eqn:Eg}:  
\begin{enumerate}
	\item[(P)] The parameters  $d>c>0$ and the function $g$, satisfying (G),  are given such that Equation \eqref{eqn:Eg} has a periodic solution $p:\mathbb{R}\to (0,\infty)$ with minimal period $\omega_p>0$, and  
	\begin{itemize}
		\item[(P1)] $p(0)=1 $, $p(t)>1$ for all $t\in [-1,0)$; 
		\item[(P2)] $(p(t),p(t-1))\ne (1,\xi_0)$ for all $t\in[0,\omega_p]$. 
	\end{itemize}
\end{enumerate}

\begin{rem}\label{remark} 
	Assume that (P) is satisfied.
	\begin{itemize}
		\item[1.] By Equation \eqref{eqn:Eg}, $p(t)=e^{-ct}$, $t\in[0,1]$. Clearly, $p_1\in C^+_*$, and then by Proposition \ref{prop:exist} and periodicity, 
		$p_t\in C^+_*$ for all $t\in\R$. In addition, 
		$\omega_p>2$ and $p$ is unique. 
		\item[2.] For given $d>c>0$ and $g$ satisfying (G), we can look for $p$ 
		in the following way.   Let $\phi(s)=e^{-c(1+s)}$, $s\in[-1,0]$, and 
		define  $x:[0,\infty)\to  \R$ by $x(t)=\Gamma(t,\phi)(-1)$. That is, 
		$x$ is a solution of \eqref{eqn:Eg} on $[0,\infty)$ with initial function $x(t)=e^{-ct}$, $t\in [0,1]$. 
		If there exists an $\omega>1$ so that $x(\omega)=1$ and $x(\omega+s)>1$ 
		for all $s\in[-1,0)$, then clearly $x(\omega+t)=e^{-ct}$ for all $t\in [0,1]$, i.e., $x_{\omega+1}=x_1$. If $\omega $ is minimal with the above property then $p$ is the $\omega$-periodic extension of 
		$x|_{[0,\omega]}$, and $\omega_p=\omega$.  
		\item[3.] By the above remark, there is no  $\omega\in (0,\omega_p)$ such that $p(\omega)=1$ and $p(\omega +s)>1$ for all $s\in[-1,0)$.
		\item[4.] Clearly, $p(t)\ge e^{-ct}$ for all $t\ge0$, and, by Proposition \ref{prop-bound}, $p(t)<d/c$ for all $t\ge 0$. 
		Therefore, $p(t)\in (e^{-c\omega_p}, d/c)$ for all $t\in\R$.  
	\end{itemize}
\end{rem}

Now we define several constants related to (P) and used in the sequel. 
Let 
$$
p_m=\min_{t\in\R }p(t),\quad p_M=\max_{t\in\R}p(t). 
$$
It is easy to see that  $0<e^{-c\omega_p}<p_m<1<p_M<d/c$. 
Choose $\kappa_1,\kappa_2$ so that 
$$
0<\kappa_1<p_m< p_M<\kappa_2.
$$
Let 
$$
m=\left[ \omega_p\right] 
$$
be the integer part of $\omega_p$. Then  $m\le \omega_p< m+1$. 

Observe that $p_0\in C^+_*$, and then $p_t\in C_*^+$ for all $t\ge 0$ by Proposition \ref{prop:exist}. This fact allows to define $k_0\in\N $ 
as the number of points in the set 
$$
\{t\in(0,\omega_p]: \, p(t)=1\}.
$$
By (P2) and continuity, a $\delta_0\in (0,\min\{\xi_0,1-\xi_0\})$ can be fixed so that 
$$
(p(t),p(t-1))\notin \left[ 1-\delta_0,1+\delta_0\right] 
\times\left[ \xi_0-\delta_0,\xi_0+\delta_0\right]
$$
for all $t\in\R$. 
By Proposition \ref{prop:mu} there exist 
$\delta_1\in (0,\delta_0)$ and $\mu>0$ such that 
$$
|-cu+dg(v)|\ge \mu \text{ for all }
(u,v)\in \left[ 1-\delta_1,1+\delta_1\right] \times 
\left(
\left[ 
0,\xi_0-\delta_0\right] \cup \left[ \xi_0+\delta_0, \infty\right)  \right).
$$  
Consequently, if $t\in\R$ and $p'(t)$ exists, then 
\begin{equation}\label{p'large}
	|p(t)-1|\le\delta_1 
	\text{ implies }|p'(t)|\ge\mu.
\end{equation} 

Define the constants 
$$
k_1=\dfrac{2(3k_0+2)}{\mu},\quad 
k_2=2+\dfrac{d}{c}+\left( d+\delta_1\right)\left( 2+2k_1 
+\|g'\|_{[\kappa_1,1)}\right).
$$
Set 
$$
\delta_2=\frac{1}{k_2^m}
\min\left\lbrace \frac{\delta_1}{2},p_m-\kappa_1, \kappa_2-p_M
\right\rbrace. 
$$
Property (P1) and continuity of $p$ imply the existence of a $\gamma\in (0,m+1-\omega_p)$ so that 
$$
p(t)>1 \text{ for all } t\in[-1-\gamma,0).
$$
Choose the  positive constants $ \varepsilon_0$ and $ \varepsilon_1 $ so that
$$\varepsilon_0\le
\min\left\lbrace 
k_2^m\delta_2, \dfrac{c}{2},
\frac{e^{c\gamma}-1}{2}, \frac{1}{2}\left( \min_{t\in[-1-\gamma,-\gamma]}p(t)-1\right) 
\right\rbrace 
\text{ and } \varepsilon_1=\frac{\varepsilon_0}{k_2^m}.
$$

The constants
\begin{equation}\label{sigma}
	\sigma_0=\frac{1}{c}\log \frac{1+2\varepsilon_0}{1+\varepsilon_0}, \quad 
	\sigma_1=\frac{1}{c}\log(1+\varepsilon_0)
\end{equation}
satisfy 
$\sigma_0+\sigma_1=(1/c)\log (1+2\varepsilon_0)\leq \gamma<m+1-\omega_p$ because of the choice of 
$\varepsilon_0$. 
Observe that $p(t)>1$, $t\in [-1-\gamma,0)$, $p(0)=1$ imply 
$p(t)=e^{-ct}$ for all $t\in [-\gamma, 1]$. 
It follows that
$$
p(-\sigma_0-\sigma_1)=1+2\varepsilon_0, \quad p(-\sigma_1)=1+\varepsilon_0.
$$ 
Let us define the shifted version of $p$ by 
$$
r:\mathbb{R}\ni t\mapsto p(t-\sigma_1)\in\mathbb{R}.
$$
Then $r$ is an $\omega_p$-periodic solution of \eqref{eqn:Eg} satisfying 
\begin{equation}\label{r-properties}
	\begin{split}
		r(t)&= (1+\varepsilon_0)e^{-ct}\quad \text{ for all } t\in[-\sigma_0,1+\sigma_1],\\
		r(-\sigma_0)&=1+2\varepsilon_0, \quad r(0)=1+\varepsilon_0, \quad r(\sigma_1)=1,\\
		r(t)&\geq 1+2\varepsilon_0  \quad \text{ for all } t\in[-1-\sigma_0,-\sigma_0].
	\end{split}
\end{equation}

For $\delta\in (0,\delta_1)$ let 
$$
\Delta(\delta)=\left\lbrace t\in [0,m]:\, |r(t)-1|<\delta\right\rbrace .
$$ 
By the continuity of $r$, $\Delta(\delta)$ is an open set in $[0,m]$, and 
$\Delta(\delta)\cap (0,m)$ is the 
disjoint union of at most countable many open intervals. 
Let $|\Delta(\delta)|$ denote  the  sum of the lengths of these disjoint open intervals.	

The next result shows that the measure of the set of those times $t$	
for which $r(t)$ is close to the discontinuity $\xi=1$ of $g$ is small. 

\begin{proposition}\label{prop:delta} 
	Suppose that (P) is satisfied. 
	Then, for all $\delta\in (0,\delta_1)$, 
	the inequality 
	\begin{equation}\label{Kdelta}
		|\Delta(\delta)|\le k_1\delta
	\end{equation} 
	holds. 
\end{proposition}

\begin{proof}  
	Let $\delta\in (0,\delta_1)$ be given, and set $\Delta=\Delta(\delta)$. 
	
	By the definition of $k_0$, there are exactly $\ell_0\in\{1,\ldots, k_0\}$ points $t_1,t_2,\dots,t_{\ell_0}$ 
	in $[0,m]$ such that $r(t_j-1)=1$, $j=1,2,\dots, \ell_0$. 
	The set  
	$$
	(0,m)\cap \Delta \setminus\{t_1,t_2,\dots,t_{\ell_0}\}
	$$
	is an open subset of $(0,m)$, and thus it is the disjoint union of at most countable many open intervals denoted by $\mathcal{I}$. 
	Clearly, the sum of the lengths of the intervals in $\mathcal{I}$ is equal to $|\Delta|$. 
	
	Let $(\alpha,\beta)$ be an element of $\mathcal{I}$. 
	By the choice of $\delta_1$ and by Propositions \ref{prop:exist},
	\ref{prop:mu} if $t\in (\alpha,\beta)$ 
	then $r'(t)$ exists, and by \eqref{p'large}
	$$
	|r'(t)|=|-cr(t)+dg(r(t-1))|\ge\mu.
	$$
	Therefore, by continuity, either  $r'(t)\ge\mu$ for all $t\in(\alpha,\beta)$, or $r'(t)\le -\mu$ for all $t\in(\alpha,\beta)$. 
	Then, for the length of the interval $(\alpha,\beta)$, one obtains
	\begin{equation}\label{length}
		\beta-\alpha\leq \frac{1}{\mu}\left|\int_\alpha^\beta r'(s)\, ds\right|
		=\frac{1}{\mu}\left|r(\beta)-r(\alpha)\right|
		\le \frac{1}{\mu}|(1+\delta)-(1-\delta)|=\frac{2\delta}{\mu}.
	\end{equation}
	
	For any $(\alpha,\beta)\in\mathcal{I}$ at least one of the five  possibilities can occur:
	\begin{itemize}
		\item[(i)]  $\alpha=0$,
		\item[(ii)]  $\beta=m$,
		\item[(iii)]  $\alpha=t_j$ for some $j\in\{1,\ldots,\ell_0\}$, 
		\item[(iv)]  $\beta=t_j$ for some $j\in\{1,\ldots,\ell_0\}$, 
		\item[(v)]  $[\alpha,\beta]\subset (0,m)$ and 
		$[\alpha,\beta]\cap\{t_1,t_2,\dots,t_{\ell_0}\}=\emptyset$.
	\end{itemize}
	Case (i) (and similarly Case (ii)) can happen for at most one interval in $\mathcal{I}$. 
	The number of intervals in $\mathcal{I}$, for which Case (iii) (and similarly  Case (iv)) can hold, is  at most $\ell_0\leq k_0$. 
	For Case (v), from the fact that either  $r'(t)\ge\mu$ for all $t\in(\alpha,\beta)$,  or $r'(t)\le-\mu$ for all $t\in(\alpha,\beta)$, 
	it is easy to see that either $r(\alpha)=1-\delta$ and $r(\beta)=1+\delta$, or $r(\alpha)=1+\delta$ and $r(\beta)=1-\delta$. 
	Thus, by continuity, there is a $t_*\in(\alpha,\beta)$ with 
	$r(t_*)=1$. By the definition of $k_0$, Case (v) may appear at most $k_0$ times.
	
	Therefore, the  set  
	$(0,m)\cap \Delta \setminus\{t_1,t_2,\dots,t_{k_0}\}$ 
	contains at most $2+2k_0+k_0=3k_0+2$ open intervals.   
	
	Finally, by applying \eqref{length}, we conclude that 
	$$
	|\Delta|\le (3k_0+2)\frac{2\delta}{\mu}=k_1\delta.
	$$ 
\end{proof}

The next result estimates the difference of the periodic solution $r$ 
of Equation \eqref{eqn:Eg} and a solution of Equation \eqref{eqn:Ef}. 

\begin{proposition}\label{prop:x-y} 
	Suppose that property (P) holds, and for some  
	$\delta\in (0,\delta_2)$ the positive constants $a, b$, and the function $f$, in addition to (F), satisfy the conditions
	\begin{itemize}
		\item[(i)] $|c-a|<\delta, \quad |d-b|<\delta$;  
		\item[(ii)] $\|f- g\|_{[\kappa_1,1-\delta]\cup[1+\delta,\kappa_2]}<\delta$.
	\end{itemize}
	Let $\psi \in C^+$ be given so that, for the solution $y=y^\psi$ 
	of \eqref{eqn:Ef}, the inequality 
	$$\|r_1-y_1\|<\delta$$  
	is valid. 
	Then 
	$$\left|r(t)-y(t)\right|<\delta k_2^m\quad \text{ for all } t\in [0,m+1].
	$$ 
\end{proposition} 

\begin{proof}
	It suffices to show that
	\begin{equation}\label{norm:x-y}
		\|r_j-y_j\|<\delta k_2^{j-1} 
	\end{equation}
	for all  $ j\in \{1,2,\dots,m+1\}$.
	
	We prove by induction. 
	According to the assumption $\|r_1-y_1\|<\delta$,  statement \eqref{norm:x-y}
	holds for $j=1$.  
	Suppose that $j\in \{1,2,\ldots,m\}$, and \eqref{norm:x-y} is satisfied. 
	It remains to show that  
	$$\|r_{j+1}-y_{j+1}\|<\delta k_2^j.$$ 
	
	By the choice of $\delta_2$, 
	$$
	\delta<\delta_2\le \frac{1}{k_2^m}\min\{p_m-\kappa_1,\kappa_2-p_M\}.
	$$
	Hence, from  the assumption $\|r_j-y_j\|<\delta k_2^{j-1} $, one finds 
	$y(t)\in[\kappa_1,\kappa_2]$ for all $t\in[j-1,j]$. 
	
	For $t\in [j,j+1]$ we have 
	$$r(t)=e^{-c(t-j)}r(j)+d\int_{j}^{t} e^{-c(t-s)}g(r(s-1)) \,ds$$ 
	and 
	$$y(t)=e^{-a(t-j)}y(j)+b\int_{j}^{t}e^{-a(t-s)}f(y(s-1))\,ds.$$ 
	Using these integral equations, $g([0,\infty))\subset[0,1]$, 
	the inequality $|e^{-cu}-e^{-a u}|\le |c-a|$ for $0\le u \le 1$, $p_M<d/c$,  and assumption (i), we obtain, for $j\leq t \leq j+1$, that
	\begin{align*}
		\left|r(t)-y(t)\right| 
		& \leq \left|e^{-c(t-j)}-e^{-a(t-j)}\right|r(j)+e^{-a(t-j)}\left|r(j)-y(j)\right| \\
		& + |d-b|\int_{j}^{t} e^{-c(t-s)}g(r(s-1)) \, ds + b\int_{j}^{t} \left|e^{-c(t-s)}-e^{-a(t-s)}\right|g(r(s-1)) \, ds\\
		& + b \int_{j}^{t} e^{-a(t-s)}\left|g(r(s-1))-f(y(s-1))\right| \, ds\\
		& \leq |c-a|r(j)+ \|r_j-y_j\|+|d-b| + (d+\delta)|c-a|\\
		& +(d+\delta)\int_{j}^{j+1} \left|g(r(s-1))-f(y(s-1))\right| \, ds\\
		& \leq \delta\left( \dfrac{d}{c}+1 +k_2^{j-1}+  (d+\delta)\right) +
		(d+\delta)\int_{j-1}^{j} \left|g(r(s))-f(y(s))\right| \,ds.
	\end{align*} 
	Let 
	$$\Delta_j(\delta)=\{t\in[j-1,j]:\quad |r(t)-1|<2\delta k_2^{j-1}\}.$$  
	Then, as $2\delta k_2^{j-1}<2\delta_2 k_2^{m}\le \delta_1$, 
	by Proposition \ref{prop:delta}, we have $|\Delta_j(\delta)|\leq 2 \delta k_2^{j-1}k_1$. 
	Hence 
	$$\int_{\Delta_j(\delta)} \left|g(r(s))-f(y(s))\right| \, ds\leq  
	|\Delta_j(\delta)|\leq 2k_2^{j-1}k_1\delta$$ 
	since $|g(r(s))-f(y(s))|\leq 1$ for all $s\in [j-1,j]$. 
	Let us define 
	$$S_+ = \{t\in[j-1,j]:\quad  r(t)\geq 1 + 2\delta k_2^{j-1}\},$$
	$$S_- = \{t\in[j-1,j]:\quad  r(t)\leq 1-2\delta k_2^{j-1}\}.$$
	From the inductive hypothesis $\|r_j-y_j\|<\delta k_2^{j-1}$, it follows that 
	$$
	1+\delta k_2^{j-1}<y(s)\le \kappa_2\quad \text{ for all } s \in S_+,
	$$
	and 
	$$
	\kappa_1\le y(s)<1-\delta k_2^{j-1}\quad \text{ for all } s \in S_-.
	$$ 
	For $s\in S_+$ we have
	\begin{align*}
		\left|g(r(s))-f(y(s))\right|& = 	f(y(s))\leq \|f\|_{[1+\delta k_2^{j-1},\kappa_2]}\leq \|f\|_{[1+\delta,\kappa_2]}\\ &=\|g-f\|_{[1+\delta,\kappa_2]} < \delta.
	\end{align*}
	If $s\in S_-$ then 
	\begin{align*}
		\left|g(r(s))-f(y(s))\right| 
		& \leq \left|g(r(s))-g(y(s))\right|+\left|g(y(s))-f(y(s))\right|\\
		& \leq \|g'\|_{[\kappa_1,1]}|r(s)-y(s)|+ \|g-f\|_{[\kappa_1,1-\delta k_2^{j-1}]}\\
		& \leq \|g'\|_{[\kappa_1,1)}\|r_j-y_j\| + \|g-f\|_{[\kappa_1,1-\delta]}\\
		& \leq \left(\|g'\|_{[\kappa_1,1)}k_2^{j-1}+1\right)\delta.
	\end{align*}
	Therefore 
	$$\int_{j}^{j+1} \left|g(r(s-1)-f(y(s-1)))\right| \,ds \leq \delta \left(2k_2^{j-1}k_1+\|g'\|_{[\kappa_1,1)}k_2^{j-1}+1\right).$$ 
	It follows that 
	\begin{align*}
		\|r_{j+1}-r_{j+1}\|
		& = \delta\left(k_2^{j-1}\left[1+(d+\delta)(2k_1+\|g'\|_{[\kappa_1,1)})\right]+\frac{d}{c}+2(d+\delta)+1\right).
	\end{align*}
	Clearly, $1\le k_2^{j-1}$, and thus
	$$\frac{d}{c}+2(d+\delta_1)+1\leq k_2^{j-1}\left(\frac{d}{c}+2(d+\delta_1)+1\right).$$
	Consequently, \begin{align*}
		\|x_{j+1}-y_{j+1}\|
		& \leq \delta k_2^{j-1}\left(2+(d+\delta_1)\left(2+2k_1+\|g'\|_{[\kappa_1,1)})+\frac{d}{c}\right)\right) = \delta k_2^j.
	\end{align*}
	This completes the proof.\end{proof}

\section{Periodic orbits}\label{sec3}

In this section we  prove that equation \eqref{eqn:Ef} has a stable periodic orbit provided property (P) holds for Equation \eqref{eqn:Eg} and 
$a, b, f$ are close to $c, d, g$ in the sense given below. 

In Section 2, for $d>c>0$ and a periodic solution $p$ of Equation \eqref{eqn:Eg} satisfying (P1) and (P2),  
we constructed the constants $p_m, p_M, \kappa_1,\kappa_2,m,k_0, \delta_0,\delta_1,\mu,k_2, \delta_2,\gamma,\varepsilon_0, \varepsilon_1, \sigma_0,\sigma_1$. 

The following result is a version of Theorem \ref{main_new} by formulating the closeness of $a, b, f$ to $c, d, g$ with the explicitly given constants $\varepsilon_0, \varepsilon_1, m$. First we prove Theorem \ref{main0}, and then show that Theorem \ref{main_new} is a consequence.

\begin{theorem}\label{main0}
	Let property (P) hold for Equation \eqref{eqn:Eg}, and suppose that  
	the positive constants $a, b$ and the function $f$ satisfy the conditions
	\begin{itemize}
		\item[(i)] $|c-a|(1+\varepsilon_0)<\frac{\varepsilon_1}{2},$
		\item[(ii)] $|d-b|<\varepsilon_1,$
		\item[(iii)] $\|f-g\|_{[\kappa_1,1-\varepsilon_1]\cup[1+\varepsilon_1,\kappa_2]}<\varepsilon_1,$
		\item[(iv)] $b\|f\|_{[1+\varepsilon_0,\kappa_2]}<\frac{\varepsilon_1}{2},$ 
		\item[(v)] $b\left(1+\frac{4b}{a}\right)\|f'\|_{[1+\varepsilon_0,\kappa_2]}\left(1+b\|f'\|_{[\kappa_1,\kappa_2]}\right)^m<1$.
	\end{itemize}
	Then equation \eqref{eqn:Ef} has a periodic solution $ q: \mathbb{R}\to\mathbb{R}$, with minimal period $\omega_q>0$ satisfying 
	$|\omega_q-\omega_p|<\frac{\varepsilon_0}{c}$, such that the periodic orbit $
	\mathcal{O}_q=\{ q_t:\, t\in [0,\omega_q]\}
	$
	is hyperbolic, orbitally stable, exponentially attractive with asymptotic phase,  and $\dist\,(\mathcal{O}_q,\mathcal{O}_p)<\varepsilon_0.$ 
	The set 
	$$ U_{\varepsilon_0,\kappa_2}= \{\psi \in C^+:\, \psi(s)\in [1+\varepsilon_0,\kappa_2] \text{ for all } s\in [-1,0]\} $$ 
	belongs to the region of attraction of $\mathcal{O}_q$. Moreover, if the stronger condition 
	$b\|f\|_{[1+\varepsilon_0,\infty)}<\frac{\varepsilon_1}{2}$ is satisfied instead of (iv), then  the region of attraction of $\mathcal{O}_q$ contains
	$$ U_{\varepsilon_0,\infty}= \{\psi \in C^+:\, \psi(s)\ge 1+\varepsilon_0 \text{ for all } s\in [-1,0]\}. $$ 
	
\end{theorem}

\begin{proof}
	
	Define 
	$$
	S_{\varepsilon_0,\kappa_2}=\{ \psi \in C:\, \psi(0)=1+\varepsilon_0, \psi(s)\in[ 1+\varepsilon_0,\kappa_2] \text{ for all } s\in [-1,0] \}.
	$$ 
	Let $b>a>0$ and $f$ be given such that conditions (i)-(v) are satisfied. Let 	  $\psi \in S_{\varepsilon_0,\kappa_2}$, and  let $y=y^\psi$ be the solution of \eqref{eqn:Ef}. 
	
	\medskip 
	
	{\it Step 1:} Proof of 	\begin{equation}\label{r-y}
		|r(t)-y(t)|<\varepsilon_1 k_2^m=\varepsilon_0\quad \text{ for all } t\in [0,m+1].
	\end{equation}	
	For $t\in [0,1]$ we have 
	$$
	r(t)=e^{-ct}r(0)=e^{-ct}(1+\varepsilon_0),
	$$ 
	$$
	y(t)=e^{-at}(1+\varepsilon_0)+b\int_{0}^{t}e^{-a(t-s)}f(y(s-1))\,ds.
	$$
	Hence, by conditions (i) and (iv), for all $t\in[0,1]$,
	\begin{align*}
		|r(t)-y(t)|&\leq |e^{-ct}-e^{-at}|(1+\varepsilon_0)+b\int_{0}^{1}f\left(\psi(s-1)\right)\,ds\\
		& \leq |c-a|(1+\varepsilon_0)+b\|f\|_{[1+\varepsilon_0,\kappa_2]}< \varepsilon_1,
	\end{align*}
	that is $\|r_1-y_1\|<\varepsilon_1$. 
	
	Observe that 	$\varepsilon_1=\varepsilon_0/k_2^m\le\delta_2$.   
	Conditions (i)--(iii) of the theorem imply (i) and (ii) of Proposition \ref{prop:x-y} with $\delta=\varepsilon_1$.  
	The assumption $\|r_1-y_1\|<\delta $ of Proposition \ref{prop:x-y} holds as well. 
	Consequently \eqref{r-y} is satisfied. 
	
	\medskip 
	
	{\it Step 2:}  Properties of $y$. 
	
	Combining properties \eqref{r-properties} of $ r $, 		
	inequality \eqref{r-y}, the choice of $\varepsilon_0$, $r(\R)=[p_m,p_M]$, $\omega_p<m+1$, it follows that 
	\begin{equation}
		\begin{split}\label{y-properties}
			& y(t)\in (r(t)-\varepsilon_0,r(t)+\varepsilon_0)\subset 
			[p_m-\varepsilon_0,p_M+\varepsilon_0]\subset (\kappa_1,\kappa_2) \quad (0 \leq t \leq m+1), \\
			& y(\omega_p-\sigma_0)>r(\omega_p-\sigma_0)-\varepsilon_0=r(-\sigma_0)-\varepsilon_0=1+\varepsilon_0,\\
			& y(\omega_p+\sigma_1)<r(\omega_p+\sigma_1)+\varepsilon_0=r(\sigma_1)+\varepsilon_0=1+\varepsilon_0,\\
			& y(t)>r(t)-\varepsilon_0\geq 1+2\varepsilon_0-\varepsilon_0=1+\varepsilon_0\quad (t \in [\omega_p-1-\sigma_0,\omega_p-\sigma_0]),\\
			& y(t)>r(t)-\varepsilon_0\geq 1- \varepsilon_0\quad ( t \in [\omega_p-\sigma_0,\omega_p+\sigma_1]).
		\end{split}
	\end{equation}
	
	\medskip 
	
	{\it Step 3:} A return map.  	
	
	By using the first and the last inequality in \eqref{y-properties} and condition (iv),  we have 
	\begin{equation}\label{y-der}
		y'(t)  = -ay(t)+bf(y(t-1))  < -a(1-\varepsilon_0)+ \frac{\varepsilon_1}{2} \ \text{ for all } t\in [\omega_p-\sigma_0,\omega_p+\sigma_1].
	\end{equation}
	From $\delta_1<\delta_0<1/2$ and $\varepsilon_0\le k_2^m \delta_2\le\delta_1/2$ one gets $\varepsilon_0<1/4$. 
	From $\varepsilon_0\le c/2$, $m\ge 1$, $k_2\ge 3$, it follows that 
	$\varepsilon_1=\varepsilon_0/k_2^m <c/6$. 
	From condition (i), clearly  $|c-a|<\varepsilon_1/2$. Therefore  
	\begin{equation}\label{a>eps1}
		a>c-\frac{\varepsilon_1}{2}>
		c- \frac{c}{12}= \frac{11c}{12}> \frac{11}{2}\varepsilon_1.
	\end{equation}
	From \eqref{y-der} with $\varepsilon_0<1/4$ and $\varepsilon_1<2a/11$, which comes from \eqref{a>eps1}, 
	we conclude 
	\begin{equation}\label{y-der2}
		y'(t) <-\frac{3}{4}a +\frac{a}{11}< -\frac{a}{2} \text{ for all } t\in [\omega_p-\sigma_0,\omega_p+\sigma_1].
	\end{equation}		
	Consequently, $y$ strictly decreases on $[\omega_p-\sigma_0, \omega_p+\sigma_1].$ 
	From \eqref{y-properties} one gets $y(\omega_p-\sigma_0)>1+\varepsilon_0$ and  $y(\omega_p+\sigma_1)<1+\varepsilon_0$. 
	Thus, by the strict monotonicity property of $y$ on $[\omega_p-\sigma_0, \omega_p+\sigma_1]$,   
	there is a unique $\tau = \tau(\psi)\in(\omega_p-\sigma_0,\omega_p+\sigma_1)$ so that $y(\tau)=1+\varepsilon_0$. 
	By \eqref{y-properties} it follows that $y_\tau \in S_{\varepsilon_0,\kappa_2}.$
	
	Now we can define a return map $R$  by 
	$$R: S_{\varepsilon_0,\kappa_2}\ni \psi\mapsto y_\tau^\psi
	=\Phi(\tau(\psi),\psi)	
	\in S_{\varepsilon_0,\kappa_2}.$$ 
	
	\medskip 
	
	{\it Step 4:} The return map $R$ is a contraction.  	
	
	Let $\psi,\chi\in S_{\varepsilon_0,\kappa_2}$ and let $y=y^\psi$, $z=z^\chi$. We want to estimate $\|R(\psi)-R(\chi)\|$. 
	
	As $\psi\in S_{\varepsilon_0,\kappa_2}$ was arbitrary in Steps 1--3, the results 
	obtained in Steps 1--3 are valid for any element of 
	$S_{\varepsilon_0,\kappa_2}$. 
	Then, by \eqref{y-properties},   $y(t),z(t)\in [\kappa_1,\kappa_2]$, $t\in[-1,m+1]$. 
	Applying Proposition \ref{prop:y-z} with $\delta=\varepsilon_0$, $\ell=m$, $\psi,\chi\in S_{\varepsilon_0,\kappa_2}$ and $\psi(0)=\chi(0)=1+\varepsilon$, 
	we obtain
	\begin{equation}\label{y-z}
		|y(t)-z(t)|\leq b\|f'\|_{[1+\varepsilon_0,\kappa_2]}\left(1+b\|f'\|_{[\kappa_1,\kappa_2]}\right)^m\|\psi-\chi\| \quad \text{ for all } t\in [0,m+1].
	\end{equation}
	
	In order to estimate $\tau(\psi)-\tau(\chi)$, observe that $y(\tau(\psi))=1+\varepsilon_0=z(\tau(\chi))$, and 
	$\tau(\psi),\tau(\chi)\in (\omega_p-\sigma_0,\omega_p+\sigma_1)$. 
	Hence and from \eqref{y-der2} and \eqref{y-z} it follows that  
	\begin{equation}
		\begin{split}\label{tau-diff}
			|\tau(\psi)-\tau(\chi)|
			& =\left|\int_{\tau(\chi)}^{\tau(\psi)} 1\,dt\right|
			<\frac{2}{a}\left|\int_{\tau(\chi)}^{\tau(\psi)}y'(t)\,dt\right|=\frac{2}{a}|y(\tau(\psi))-y(\tau(\chi))|\\
			&= \frac{2}{a}|z(\tau(\chi))-y(\tau(\chi))|\leq \frac{2b}{a}\|f'\|_{[1+\varepsilon_0,\kappa_2]}\left(1+b\|f'\|_{[\kappa_1,\kappa_2]}\right)^m\|\psi-\chi\|.
		\end{split}
	\end{equation}
	By Proposition \ref{prop-bound}, $y(t)<b/a$ for all $t\ge 0$, 
	and  
	then $|y'(t)|=|-ay(t)+bf(y(t-1))|<2b$ for all $t>0$. Thus $\|y_{\tau_y}-y_{\tau_z}\|\leq 2b|\tau_y-\tau_z|.$ Hence
	\begin{equation}\label{2b}
		\|\Phi(\tau(\psi),\psi)-\Phi(\tau(\chi),\psi)\|
		\le 2b |\tau(\psi)-\tau(\chi)|.
	\end{equation}	
	Combining  \eqref{2b}, \eqref{tau-diff} and \eqref{y-z} it follows that  	
	\begin{align*}		
		\|R(\psi)-R(\chi)\|
		&=\|\Phi(\tau(\psi),\psi)-\Phi(\tau(\chi),\chi)\|\\
		&\leq \|\Phi(\tau(\psi),\psi)-\Phi(\tau(\chi),\psi)\|
		+\|\Phi(\tau(\chi),\psi)- \Phi(\tau(\chi),\chi)\|\\
		&\leq 2b |\tau(\psi)-\tau(\chi)| 
		+
		b\|f'\|_{[1+\varepsilon_0,\kappa_2]}\left(1+b\|f'\|_{[\kappa_1,\kappa_2]}\right)^m\|\psi-\chi\|\\
		& \leq b\left(1+\frac{4b}{a}\right)\|f'\|_{[1+\varepsilon_0,\kappa_2]}\left(1+b\|f'\|_{[\kappa_1,\kappa_2]}\right)^m\|\psi-\chi\|=\kappa\|\psi-\chi\|,
	\end{align*}
	where
	$$
	\kappa=b\left(1+\frac{4b}{a}\right)\|f'\|_{[1+\varepsilon_1,\kappa_2]}\left(1+b\|f'\|_{[\varepsilon_1,\kappa_2]}\right)^m<1
	$$ 
	by condition (v). 
	Therefore, $R:S_{\varepsilon_0,\kappa_2}\to S_{\varepsilon_0,\kappa_2}$ is a contraction. 
	
	\medskip 
	
	{\it Step 5:} The periodic orbit. 
	
	The closed subset $S_{\varepsilon_0,\kappa_2}$ of $C$ is a complete metric space with the metric induced by the norm of $C$. By Step 4, 
	$R: S_{\varepsilon_0,\kappa_2}\to S_{\varepsilon_0,\kappa_2}$ is a contraction, and has a unique fixed point $\phi^*$ in $S_{\varepsilon_0,\kappa_2}$ satisfying
	$$ 
	\phi^*=\Phi(\tau(\phi^*),\phi^*)
	$$
	with $\tau(\phi^*)\in [\omega_p-\sigma_0,\omega_p+\sigma_1]$.
	Then, for the solution $y^*=y^{\phi^*}:[-1,\infty)\to\R$ of Equation \eqref{eqn:Ef}, $y^*(t)=y^*(t+\tau(\phi^*))$ for all $t\ge -1$. 
	Let $\omega_q=\tau(\phi^*)$, and let $q:\R\to\R$ be the $\omega_q$-periodic extension of $y^*$. 
	Clearly, $q$ is an $ \omega_q $-periodic solution of  Equation \eqref{eqn:Ef}. 
	
	As $\phi^*\in S_{\varepsilon_0,\kappa_2}$, the estimations obtained for 
	$y=y^\psi$ in Steps 1--3 are valid for $y^*$. 
	Therefore, 
	\begin{equation}
		\begin{split}\label{q-properties}
			&|q(t)-r(t)|
			<\varepsilon_0 \quad (t\in [0,m+1]),\\
			& q(t)> 1+\varepsilon_0 \quad (t\in [\omega_p-1-\sigma_0,\omega_p-\sigma_0),\\
			& q \text{ strictly decreases on } [\omega_p-\sigma_0,\omega_p+\sigma_1].
		\end{split}
	\end{equation}
	It easily follows that $\dist\,(\mathcal{O}_q,\mathcal{O}_p)<\varepsilon_0$, and 
	$$
	|\omega_q-\omega_p|<\max\{\sigma_0,\sigma_1\}\le \frac{\varepsilon_0}{c}.
	$$
	The exponential attractivity and asymptotic phase properties of the solution $q$ are consequences of the contractivity of $R$ in the same way as in  Chapter XIV of \cite{DVLGW}, and \cite{BKV}.
	
	\medskip 
	
	{\it Step 6:} Region of attraction. 
	
	As $ R $ is a contraction on the complete metric space 
	$ S_{\varepsilon_0,\kappa_2} $, the set $ S_{\varepsilon_0,\kappa_2} $ belongs to the region of attraction of $\mathcal{O}_q$. 
	
	For $\psi \in U_{\varepsilon_0,\kappa_2}$, we claim that there is 
	a $t_*\ge 0$ such that $y_{t_*}^\psi\in S_{\varepsilon_0,\kappa_2} $. 
	
	If $\psi(0)=1+\varepsilon_0$, then the claim holds with $t_*=0$. 
	Suppose $\psi(0)>1+\varepsilon_0$, and let $y=y^\psi$. 
	Choose a maximal $T\in (0,\infty)\cup\{\infty\}$ so that $y(t)>1+\varepsilon_0$ for all $t\in (0,T)$.  
	
	Define the sequence $(t_k)$ by $t_0=0$, $t_k=\min\{t_{k-1}+1,T\}$, $k\in \N$.  
	By using $y_0=\psi \in U_{\varepsilon_0,\kappa_2}$, 
	condition (iv), and \eqref{a>eps1}, we obtain that 
	$$
	y'(t)=-ay(t)+bf(\psi(t-1))
	<-a(1+\varepsilon_0)+\frac{\varepsilon_1}{2}
	<-\frac{11}{2}\varepsilon_1+\frac{\varepsilon_1}{2}=-5\varepsilon_1 
	$$
	for all $t\in (0,t_1]$. 
	Hence, $y$ strictly decreases on $[t_0,t_1]$, and  either $y_{t_1}\in S_{\varepsilon_0,\kappa_2} $, or 
	$y(t_1)>1+\varepsilon_0$ and $y_{t_1}\in U_{\varepsilon_0,\kappa_2}$. 
	If $k\in\N$ and $y(t_k)>1+\varepsilon_0$ and $y_{t_k}\in U_{\varepsilon_0,\kappa_2}$, then  
	in the same way as above,  
	$y'(t)<-5\varepsilon_1$ follows for all $t\in (t_k,t_{k+1}]$.  
	Repeating this process, we either find a $k\in\N$  so that the claim holds 
	with $t_*=t_k$, or 
	$$
	y(t)>1+\varepsilon_0,\quad y'(t)<-5\varepsilon_1 
	\text{ for all }t>0, 
	$$ 
	which is a contradiction. 
	
	Therefore, the claim holds, and consequently 
	$ U_{\varepsilon_0,\kappa_2} $ is also in the region of attraction of $\mathcal{O}_q$. 	
	
	If $\chi\in U_{\varepsilon_0,\infty}$ and $b\|f\|_{1+\varepsilon_0,\infty)}<\varepsilon_1/2$ holds,  
	then following the proof of the above claim, a $t_*\ge 0$ can be obtained so that $y^\chi(t_*)=1+\varepsilon_0$ and $y^\chi(t_*+s)\ge 1+\varepsilon_0$, 
	$-1\le s\le 0$. 
	Choose $\psi=y_{t_*}^\chi$ as an initial function of the solution 
	$y=y^\psi $ of Equation \eqref{eqn:Ef}. Then, although $\psi$ may not be in $S_{\varepsilon_0,\kappa_2}$, the stronger condition  $b\|f\|_{1+\varepsilon_0,\infty)}<\varepsilon_1/2$ instead of (iv) allows to apply the above Steps 1--3 to find a $\tau>0$ with $y_\tau^\psi\in S_{\varepsilon_0,\kappa_2}$. 
	Therefore, $ U_{\varepsilon_0,\infty} $ belongs to the region of attraction of $\mathcal{O}_q$. 	
	
	\medskip 
	
	{\it Step 7:} $\omega_q$ is the minimal period of $q$.  
	
	Assume that $\overline{\omega_q}\in (0,\omega_q)$ is also a period of $q$, that is 
	$q_0=q_{\overline{\omega_q}}$. 
	
	If $\overline{\omega_q}\in (0,1]$ then $q(t)\ge 1+\varepsilon_0$ follows 
	for all $t\in \R$. This is impossible by \eqref{q-properties} since 
	$r(t)<1$ for $t\in (\sigma_1,1+\sigma_1]$. 
	The properties \eqref{q-properties} of $q$ and   $\omega_q\in [\omega_p-\sigma_0,\omega_p+\sigma_1]$ imply that   $\overline{\omega_q}<\omega_p-1-\sigma_0$. 
	Therefore, 
	$$
	1<\overline{\omega_q}<\omega_p-1-\sigma_0,
	$$
	and 
	$q(\overline{\omega_q})=1+\varepsilon_0$,
	$q(\overline{\omega_q}+s)\ge 1+\varepsilon_0$, $-1\le s \le 0 $.
	
	Applying \eqref{q-properties}, it follows that 
	$r(t)>1$ for all $t\in [\overline{\omega_q}-1,\overline{\omega_q}]$, and $r(\overline{\omega_q})\in (1,1+2\varepsilon_0)$. 
	Then, from Equation \eqref{eqn:Eg}, $r(t)=r(\overline{\omega_q})e^{-c(t-\overline{\omega_q})}$, 
	$t\in[\overline{\omega_q},\overline{\omega_q}+1]$. Hence, by the choice of $\varepsilon_0$ and 
	$\gamma\in (0,1)$,  
	we get 
	$r(\overline{\omega_q}+1)<(1+2\varepsilon_0)e^{-c}<1$. 
	Then there is a $\nu\in (\overline{\omega_q},\overline{\omega_q}+1)$ such that 
	$$  
	r(\nu)=1,\quad r(\nu +s)>1 \text{ for all }s\in [-1,0).
	$$
	By Remark \ref{remark} and the definition of $r$, $\nu-\sigma_1$ is a period 
	of $p$ . 
	Clearly, 
	$$
	\nu-\sigma_1\in (1-\sigma_1,\omega_p-\sigma_0-\sigma_1)\subset (0,\omega_p),
	$$
	a contradiction to the minimality of $\omega_p$. 
	Consequently, $\omega_q$ is the minimal period of $q$.    	
\end{proof}

In the Introduction we formalized Theorem \ref{main_new} in order to state the main result in a not too  technical way. 
Now we show how to obtain  Theorem \ref{main0} from Theorem \ref{main_new}.

\begin{proof}[Proof of Theorem \ref{main_new}]
	Assume property (P), and choose $\kappa_1,\kappa_2,m,\delta_0,\delta_1,\mu,k_0,k_1,k_2,\gamma,\varepsilon_,\varepsilon_1$ as in Section 2. 
	
	Let $\varepsilon\in (0,1)$ be so small  that
	$$g(1-\varepsilon)>4\varepsilon,$$ 
	$$\varepsilon<
	\min\left\lbrace 
	\frac{\varepsilon_1}{2(1+\varepsilon_0)}, 
	\frac{\varepsilon_1}{2(d+\varepsilon_1)}
	\right\rbrace 
	$$
	and
	\begin{equation}\label{eps-abm}
		\varepsilon< \frac{1}{(d+\varepsilon_1)\left(  1+4\frac{d+\varepsilon_1}{c-\varepsilon_1}\right)(1+d+\varepsilon_1)^m}.
	\end{equation}
	
	Suppose that $a,b,f$ are given so that conditions (1),(2),(3)
	of Theorem \ref{main_new} are satisfied. 
	Then  conditions (i),(ii) and (iii) of 
	Theorem \ref{main0}  hold since $\varepsilon<\frac{\varepsilon_1}{2(1+\varepsilon0)}<\varepsilon_1$. 
	
	Combining the fact $g(\xi)=0$ for $\xi>1$, condition (2) of Theorem \ref{main_new}, $\varepsilon_1<\varepsilon_0$, 
	the choice of $\varepsilon$, we obtain 
	$$ b\|f\|_{[1+\varepsilon_0,\kappa_2]}\le 
	b\|f\|_{[1+\varepsilon,\kappa_2]}
	<b\varepsilon<b\frac{\varepsilon_1}{2(d+\varepsilon_1)}
	<\frac{\varepsilon_1}{2},  
	$$ 
	that is, condition (iv) in Theorem \ref{main0} is valid.

	In order to show (v) of Theorem \ref{main0}, from $\varepsilon<\varepsilon_1<\varepsilon_0$ obviously 
	$\|f'\|_{[1+\varepsilon_0,\kappa_2]}\leq \|f'\|_{[1+\varepsilon,\kappa_2]}$.

	Observe that  $\|f-g\|_{[\kappa_1,1-\varepsilon]}<\varepsilon$ and 
	$g(1-\varepsilon)>4\varepsilon$ imply 
	$f(1-\varepsilon)> 3\varepsilon$. 
	From $\|f\|_{[1+\varepsilon,\kappa_2]}=\|f-g\|_{[1+\varepsilon,\kappa_2]}<\varepsilon$ one gets $f(1+\varepsilon)<\varepsilon$. 
	Consequently, there is a $\xi\in (1-\varepsilon,1+\varepsilon)$ such that 
	$f'(\xi)<-2$, 
	and thus  $\|f'\|_{[\kappa_1,\kappa_2]}>1$.  
	
	From properties (i), (ii) of Theorem \ref{main0}, and from  \eqref{eps-abm} one has
	\begin{equation}\label{eps-abm2}
		b\left(1+\frac{4b}{a}\right)(1+b)^m \varepsilon
		<
		(d+\varepsilon_1)\left(  1+4\frac{d+\varepsilon_1}{c-\varepsilon_1}\right)
		(1+d+\varepsilon_1)^m\varepsilon <1.
	\end{equation}
	
	Assuming (3) of Theorem \ref{main_new}, using the above results and  
	\eqref{eps-abm2}, it follows that
	$$
	b \left(1+\frac{4b}{a}\right)\|f'\|_{[1+\varepsilon_0,\kappa_2]}\left(1+b\|f'\|_{[\kappa_1,\kappa_2]}\right)^m  
	\leq b\left(1+\frac{4b}{a}\right)\|f'\|_{[1+\varepsilon,\kappa_2]}\left(1+b\|f'\|_{[\kappa_1,\kappa_2]}\right)^m
	$$
	$$
	= b\left(1+\frac{4b}{a}\right) \frac{\left(1+b\|f'\|_{[\kappa_1,\kappa_2]}\right)^m}{\left(\|f'\|_{[\kappa_1,\kappa_2]}\right)^m} \|f'\|_{[1+\varepsilon,\kappa_2]}\left(\|f'\|_{[\kappa_1,\kappa_2]}\right)^m
	$$
	$$
	< b\left(1+\frac{4b}{a}\right)\left(\frac{1}{\|f'\|_{[\kappa_1,\kappa_2]}}+b\right)^m\varepsilon  <b\left(1+\frac{4b}{a}\right)(1+b)^m \varepsilon<1,
	$$
	that is (v) of Theorem \ref{main0} holds. 
	This completes the proof.
\end{proof}

\subsection{Examples}\label{subsec2}
We show that if $f$ is the nonlinearity given in \eqref{proto}, and  $g(\xi)=\xi^k$, $0\le\xi\le 1$, 
$g(\xi)=0$, $\xi>1$, then, for any given $m\in\N$ and $\varepsilon,\kappa_1,\kappa_2$ with $0<\kappa_1<1-\varepsilon<1+\varepsilon<\kappa_2$, conditions (2) and (3) of Theorem \ref{main_new} 
hold provided $n$ is sufficiently large. 
Consequently, by the proof  of Theorem \ref{main_new}, if (P) holds and $a,b$ are sufficiently close to $c,d$ then the conditions of Theorems \ref{main_new} and \ref{main0} 
are satisfied for the prototype nonlinearity \eqref{proto} provided 
$n$ is large enough.


Let $k>0$ be fixed, and for the parameter $n\ge k$ consider 
$$f_n(\xi)=\frac{\xi^k}{1+\xi^n}\quad (\xi\ge 0), $$ 
and
$$
g(\xi)=\begin{cases}
	\xi^k&\text{ if } \xi\in[0,1],\\
	0&\text{ if } \xi>1.
\end{cases}
$$
Function $g$ satisfies condition (G), and 
$$
\lim_{n\to\infty}f_n(\xi)=g(\xi)
$$
for all $\xi \in[0,\infty)\setminus \{1\}$.

\begin{claim*} 
	Let $m\in\N$ and $\varepsilon,\kappa_1,\kappa_2$ be given constants such that 
	$$
	0<\kappa_1<1-\varepsilon<1+\varepsilon<\kappa_2.
	$$
	There exists  an 
	$N>0$ so that for  the function $f=f_n$ assumption (F) 
	and conditions (2),(3) of Theorem \ref{main_new} are satisfied provided  $n\geq N$.
\end{claim*}

\begin{proof}
	It is clear that $f_n(\xi)\leq 1$ for all $\xi\geq 0$, and $f_n(\xi)<1$ for $\xi>1$ since $n\ge k$.
	
	For condition (2) of Theorem \ref{main_new}
	we have to estimate $|g(\xi)-f_n(\xi)|$. 
	If $\xi\le1-\varepsilon$, then 
	$$
	|g(\xi)-f_n(\xi)|=\left|\xi^k \frac{\xi^n}{1+\xi^n}\right|
	\leq \xi^n\le (1-\varepsilon)^n<\varepsilon ,\quad n>n_1(\varepsilon).
	$$ 
	provided $n\ge N_1$ for some $N_1>0$. 
	
	If $\xi\geq 1+\varepsilon$ then  
	$$|g(\xi)-f_n(\xi)|=
	f_n(\xi)\leq\frac{\xi^k}{1+\xi^k \xi^{n-k}}=\frac{1}{\frac{1}{\xi^k}+\xi^{n-k}}<\frac{1}{\xi^{n-k}}<\frac{1}{(1+\varepsilon)^{n-k}}<\varepsilon
	$$
	whenever $n\ge N_2$ for some $N_2\ge k$. 
	
	As $f_n'(\xi)=\frac{k\xi^{k-1}}{1+\xi^n}-\frac{n\xi^k \xi^{n-1}}{(1+\xi^n)^2}$,  for $\xi>0$, by $n\ge k$  we have 
	$$
	|f'_n(\xi)|\leq \xi^{k-1}(k+n)
	\leq 2n \max\{\kappa_1^{k-1},\kappa_2^{k-1}\}=K_0 n  \quad (\xi \in [\kappa_1,\kappa_2]),
	$$ 
	where $K_0=2 \max\{\kappa_1^{k-1},\kappa_2^{k-1}\}$. 
	For $\xi\in [1+\varepsilon,\kappa_2]$
	\begin{align*}
		|f_n'(\xi)| & \leq \frac{1}{1+\xi^n}\left| k \xi^{k-1}-n\xi^{k-1} \frac{\xi^n}{1+\xi^n}\right|\\
		& \leq \frac{1}{1+\xi^n}(n+k)\xi^{k-1}\leq (n+k) \max\{1,\kappa_2^{k-1}\} \frac{1}{1+(1+\varepsilon)^n}\\
		& \leq K_0   \frac{n}{1+(1+\varepsilon)^n}.\\
	\end{align*} 
	Condition (3) requires 
	$$
	\|f'\|_{[1+\varepsilon,\kappa_2]}\left(\|f'\|_{[\kappa_1,\kappa_2]}\right)^m<\varepsilon
	$$
	which follows from
	$$
	\dfrac{(K_0 n)^{m+1}}{1+(1+\varepsilon)^n}<\varepsilon.
	$$
	It is elementary  that the last inequality holds if $n\ge N_3$ 
	for some sufficiently large $N_3\ge k$. 
	
	Choosing $N=\max\{N_1,N_2,N_3 \}$, the Claim is satisfied. 
\end{proof}

Notice that if (P) holds, then the constants $p_m$, $p_M$, $\kappa_1$,$\kappa_2$, $m$, $k_0$, $\delta_0$, $\delta_1$,$\mu$, $k_2$, $\delta_2$,$\gamma$, $\varepsilon_0$, $\varepsilon_1$, $\sigma_0$, $\sigma_1$ can be  explicitely constructed in terms of  
the properties of the periodic solution $p$ of Equation \eqref{eqn:Eg}. 
The constant $\varepsilon >0$ stated in Theorem \ref{main_new} is given 
as a function of $g$, $c$, $d$, $\varepsilon_0$, $\varepsilon_1$ and $m$, see the proof of  Theorem \ref{main_new}. 
Consequently, a threshold  $N$ can be explicitely given from the properties of $p$ in (P) such that a stable periodic orbit of Equation \eqref{eqn:Ef}  exists with the 
nonlinearity $f(\xi)=\frac{\xi^k}{1+\xi^n}$ for all $n\ge N$, provided 
$a,b$ are close enough to $c,d$, respectively.


\bigskip

For a fixed $k>0$, another example for a pair of functions 
$f=\tilde f_n$, $g$ is 
$$
\tilde f_n(\xi)=\xi^k e^{-x^n} \quad (\xi\ge 0), \qquad
g(\xi)=\begin{cases}
	\xi^k&\text{ if } \xi\in[0,1],\\
	0&\text{ if } \xi>1
\end{cases}
$$
with parameter $n>0$. 
The above Claim is valid for $f=\tilde f_n$ as well, provided $n$ is suffieciently large. The proof is analogous to the above one, so it is omitted. 

\section{Property (P): an analytic proof}\label{sec4}

We show that if $g$ satisfies (G) and $d$ is large compared to $c$, then property (P) holds. 
Define 
$$
I_{cg}=\int_1^2 e^{cs}g\left( e^{-c(s-1)}\right)\, ds.
$$

\begin{theorem}\label{verifyP}
	If $c>0$, $d>0$ and $g$ are given so that (G) and 
	$$
	d>\max \left\{ \dfrac{e^{2c}-1}{I_{cg}}, \dfrac{c}{g(e^{-c})}\right\}
	$$
	are satisfied, then property (P) holds. 
\end{theorem}

First we show the following statement. 

\begin{lemma}\label{lemma1}
	If $p$ is a periodic solution of Equation \eqref{eqn:Eg} with property 
	(P1), and $dg(e^{-c})>c$, then $p(t)\ge e^{-c}$ for all $t\in\R$. 
\end{lemma}

\begin{proof}
	As $p(t)>1$ for $t\in[-1,0)$ and $p(t)=e^{-ct}$ for $t\in[0,1]$, 
	it is sufficient to prove that, for any $t_0\ge 1$, 
	if 
	$$
	p(t)\ge e^{-c} \text{ for all } t\in[0,t_0],
	$$
	then 
	$$
	p(t)>e^{-c} \text{ for all } t\in (t_0,t_0+1].
	$$
	Apply \eqref{inteq-ei} with $\tau=t_0$ and $t\in [t_0,t_0+1]$, and use 
	$dg(e^{-c})>c$ to get
	\begin{align*}
		p(t) & = e^{-c(t-t_0)}p(t_0) + d\int_{t_0}^t e^{-c(t-s)}g(p(s-1))\, ds\\
		& \ge e^{-c(t-t_0)}e^{-c} + d\int_{t_0}^t e^{-c(t-s)}g(e^{-c})\, ds\\
		& > e^{-c(t-t_0)}e^{-c} + \int_{t_0}^t c e^{-c(t-s)}\, ds\\
		& = 1- e^{-c(t-t_0)}[1-e^{-c}] \\
		& > 1- [1-e^{-c}]=e^{-c},
	\end{align*}
	and the proof is complete. 
\end{proof}

\begin{proof}[Proof of Theorem \ref{verifyP}]
	In order to find a periodic solution $p$ of \eqref{eqn:Eg} satisfying (P1), it is sufficient to find a solution 
	$x:[-1,\infty)\to\R$ of  \eqref{eqn:Eg} such that 
	$x(t)=e^{-ct}$ for $t\in[0,1]$, and there is a $t_1>1$ with $x(t_1)=1$ and $x(t)>1$ for $t\in[t_1,t_1+1]$. Indeed, in this case $x_{t_1+1}=x_1$, and a
	periodic extension of $x|_{[0,t_1]}$ defines a $t_1$-periodic 
	$p:\R\to\R$ with (P1).  
	
	Suppose  $x(t)=e^{-ct}$ for $t\in[0,1]$. 
	Then, for all $t\in(1,2)$, $x(t-1)=e^{-c(t-1)}>e^{-c}$, and, by \eqref{eqn:Eg} and condition $dg(e^{-c})>c$, we get
	\begin{equation}\label{x'positive}
		x'(t)  = -cx(t)+dg(x(t-1)) > -cx(t)+dg(e^{-c})>-cx(t)+c.
	\end{equation}	
	Thus, if $t\in(1,2)$ and $x(t)=1$, then $x'(t)>0$. 
	Consequently, there is at most one $t_0\in(1,2)$ with $x(t_0)=1$. 
	For the existence of such a $t_0$, apply \eqref{inteq-ei} with $\tau=0$ and  $t\in[1,2]$ to find 
	$$
	x(t)  = e^{-c(t-1)}x(1)+d\int_{1}^{t}e^{-c(t-s)}g(e^{-c(s-1)})\, ds,
	$$
	and
	$$
	x(2)=e^{-2c}\left[1+dI_{cg}\right].
	$$
	Condition $d>\frac{e^{2c}-1}{I_{cg}}$ yields  $1+dI_{cg}>e^{2c}$ and $x(2)>1$. 
	Hence $x(1)=e^{-c}<1<x(2)$, and we conclude that there exists a unique 
	$t_0\in (1,2)$ with $x(t_0)=1$. 
	
	From the uniqueness of $t_0\in (1,2)$ with $x(t_0)=1$ it can be obtained 
	that  $x(t)>1$ for all $t\in (t_0,2]$. In addition, we clearly have 
	$x(t)\in [e^{-c},1]$ for all $t\in [1,t_0]$. 
	
	Apply \eqref{inteq-ei} with $\tau=2$ and  $t\in (2,t_0+1]$, use $x(2)=
	e^{-2c}[1+I_{cg}]$, and 
	$x(t)\in [e^{-c},1]$ for $t\in [1,t_0]$, and 
	the conditions  $dg(e^{-c})>c$, $1+dI_{cg}>e^{2c}$ to get
	\begin{align*}
		x(t) & = e^{-c(t-2)}x(2)+d\int_{2}^{t}e^{-c(t-s)}g(x(s-1))\, ds\\
		& \ge e^{-ct}\left[ 1+dI_{cg}\right] 	+ e^{-ct}\int_2^t e^{cs}dg(e^{-c})\, ds\\
		& > e^{-ct}\left[ 1+dI_{cg}\right]+ e^{-ct}\int_2^t e^{cs}c\, ds\\
		& = 1 +e^{-ct}\left[ 1+dI_{cg}-e^{2c}\right] >1.
	\end{align*}
	
		Therefore, $x(t)>1$ for all $t\in (t_0,t_0+1]$. 
		Then, clearly, there exists a $t_1>t_0+1$ such that 
		$$
		x(t_1)=1,\quad x(t)>1 \text{ for all }t\in [t_1-1,t_1).
		$$
		Then, by Equation \eqref{eqn:Eg}, $x(t)=e^{-c(t-t_1)}$, 
		$t\in [t_1,t_1+1]$, 
		that is 
		$$x_{t_1}=x_0\quad\text{and}\quad x_{t_1+1}=x_1.$$
		so the restriction of $x$ to $ [0,t_1] $ can be extended to a $t_1$-periodic solution $p:\R\to\R$ of Equation \eqref{eqn:Eg} with minimal period 
		$\omega_p=t_1>2$. Obviously, (P1) holds. 
		
		Property (P2) is a consequence of Lemma \ref{lemma1} and 
		$$
		p(t-1)\ge e^{-c}>g^{-1}\left(\dfrac{c}{d}\right)=\xi_0 \quad (t\in\R).
		$$  
	\end{proof}

	In the case $g(\xi)=\xi^k$, $0\le\xi\le1$, for some $k>0$, 
	the expression $I_{cg}$  has a simple form:
	$$
	I_{cg}=\begin{cases}
		e^c&\text{ if } k=1,\\
		{e^c}\frac{1-e^{-(k-1)c}}{(k-1)c}
		&\text{ if } k\in(0,1)\cup(1,\infty).
	\end{cases}
	$$
	
	\begin{corollary}
		Let $k>0$,   $g(\xi)=\xi^k$ for $0\le\xi\le1$, and $g(\xi)=0$ for $\xi>0$. 
		Then property (P) holds if 
		$$
		d>\max\{e^c-e^{-c},ce^c\}\ \text{ in case }\ k=1,
		$$
		and 
		$$
		d > \max \left\{  (e^c-e^{-c})
		\frac{(k-1)c}{1-e^{-(k-1)c}}, ce^{kc} 
		\right\}\ \text{ in case }\ k\ne 1.
		$$
	\end{corollary}
	
	The above condition for $k=1$ was obtained in \cite{BKV}. For $k=2$, 
	$$d>\max\left\lbrace c(e^c+1),ce^{2c}\right\rbrace;$$ and for $k=\frac{1}{2}$,  
	$$d>\max\left\lbrace \frac{c}{2}(1+e^{-c})(1+e^{c/2}),ce^{c/2}\right\rbrace$$
	guarantee property (P). 
	
	\begin{figure}[h]
		\centering
		\includegraphics[width=1\linewidth]{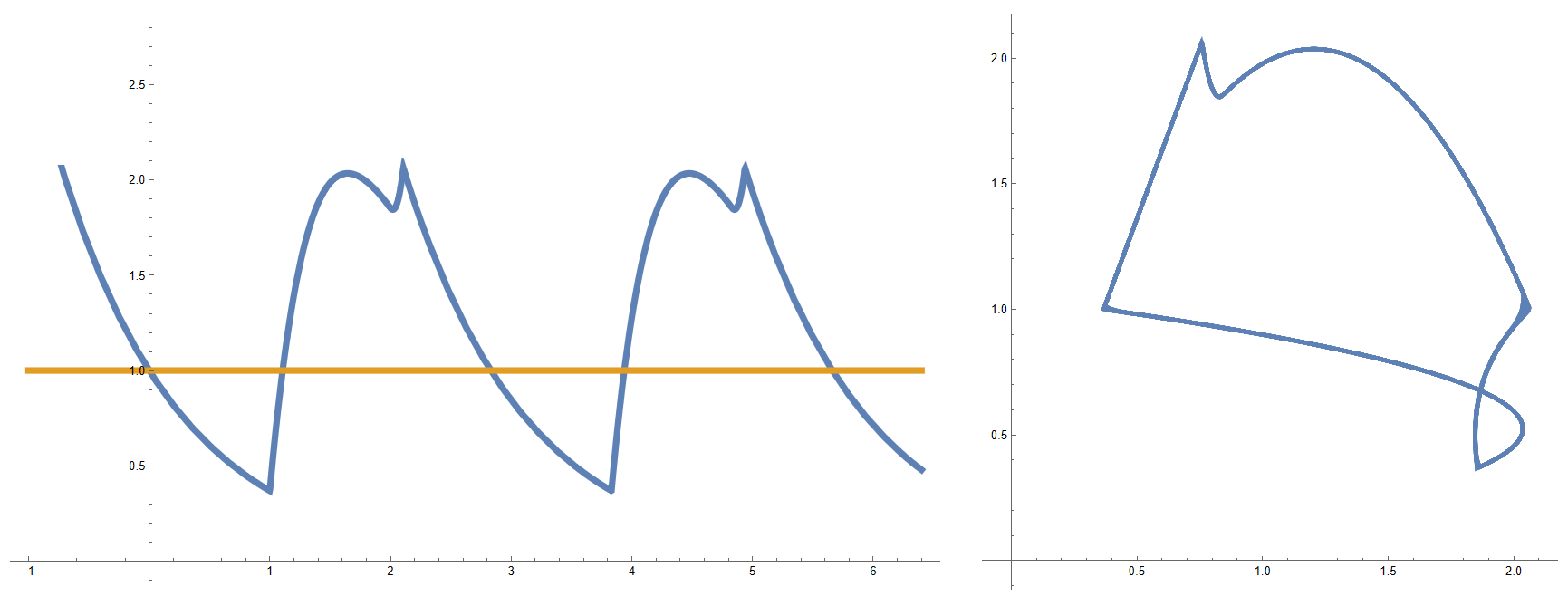}
		\caption{Plots of solution and their projection with the map $t\mapsto(p(t),p(t-1))$ in the case $k=2$ with parameter values $c=1$ and $d=7.38907)$.}
		\label{fig:figk2}
	\end{figure}

	\begin{figure}[h]
		\centering
		\includegraphics[width=1\linewidth]{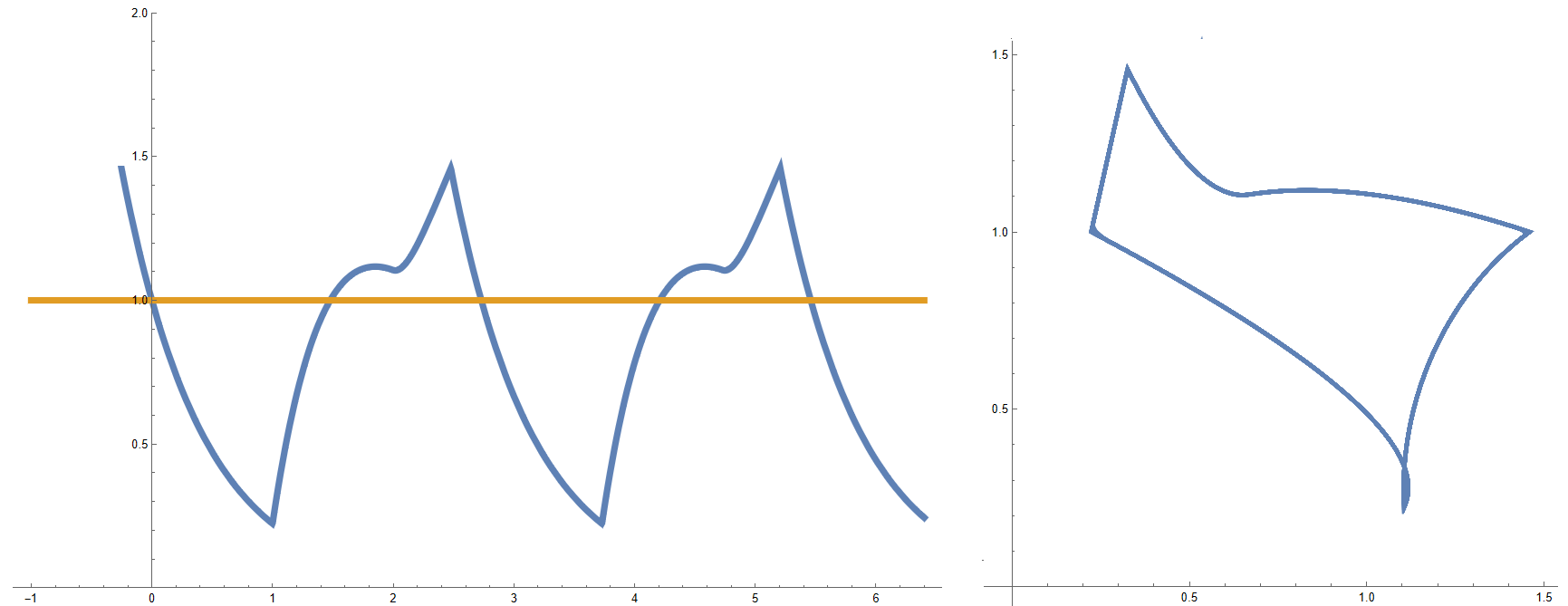}
		\caption{Plots of solution and their projection with the map $t\mapsto(p(t),p(t-1))$ in the case  $k=\frac{1}{2}$ with parameter values $c=1.5$, and $d=3.1756$.}
		\label{fig:figk0.5}
	\end{figure}

\section{Property (P): computer assisted proofs}\label{CAP}
We assume the reader is familiar with work
\cite{BKV}. In what follows, similar to the
aforementioned work, we will develop a method to prove,
for a fixed $g$ and given parameters $c, d$ that the
equation \eqref{eqn:Eg} has a periodic orbit as in the condition (P).
The method from \cite{BKV}
would not work for functions $g$ used in this paper, therefore we
use a general purpose algorithm for integration forward in time of
DDEs from papers \cite{SZ}, \cite{SZ2}. The method is general in the sense
that any form of $g$ given in terms of simple functions
(composition of arithmetic operators and some analytic functions) can be used.
We present examples for $g$ obtained as a limit $n \to \infty$
from our prototype nonlinearity
\eqref{proto} with $k=2$ and $k=1/2$.

Since \eqref{eqn:Eg} is discontinuous, the direct application
of the algorithm from \cite{SZ2} is not possible
and we adopt a technique  similar to that of \cite{BKV} 
to generate the segments of the solution that are smooth on explicitly
given intervals. In short, we are computing images of the constitutive 
Poincar\'e maps to sections $S_0 = \{\phi \in C([-1, 0], \R^+) : x(0) = 1\}$
and $S_1 = \{\phi \in C([-1, 0], \R^+) : x(-1) = 1\}$. The algorithm
from \cite{SZ2} will take care of providing 
the rigorous estimates on the solutions, crossing times and  
the transversality of the intersection with the section. 

We would like to acknowledge that any algorithm that can continue
solutions of DDEs forward in time can be used, as long as it can 
(a) do rigorous method of steps over full delay intervals 
and (b) it can compute intersection of
the solutions with the sections $S_0$ and $S_1$.  
Several alternatives exists in the literature:
\cite{C, LJ, RA}. The advantage of using  \cite{SZ2}
is that it provides Poincar\'e maps (point (b)) without additional
effort and it is relatively easy to obtain
estimates on segments $x_{t}$ of the solution
without the assumption that $t$ is a multiple of the time lag.

We use the following notions from \cite{BKV}.
By $\IR$ we denote the set of all closed intervals in $\R$.
Let $\phi : D \to \R$ and $[\alpha, \beta] \subset D$. Then:
\begin{align*}
	\phi_{[\alpha, \beta]} : [0, \beta - \alpha] \to \R & 
	\quad\quad \phi_{[\alpha, \beta]}(s) := \phi(\alpha + s) \\ 
	\phi_{[\alpha, \beta)} : [0, \beta - \alpha) \to \R & 
	\quad\quad \phi_{[\alpha, \beta)}(s) := \phi(\alpha + s)
\end{align*}
We define the following spaces:
\begin{align*}
	C^{>1}([0, \delta]) &= \left\{ \phi \in C([0, \delta],[1,\infty)): \forall_{s \in (0, \delta)}\ \phi(s) > 1 \right\}, \\
	C^{\le 1}([0, \delta]) &= C^\infty([0, \delta], [0,1]), \\
	C^{+}_{seg}([0, \delta]) &= C^{>1}([0, \delta]) \cup C^{\le 1}([0, \delta]), \\
	C^{+}_{r,sol} &= \left\{ \phi \in C^{+}_{r} : \exists \{ d_i \}_{i=0}^{m} \subset [-1, 0], -1 = d_0 < \ldots < d_m = 0, \right. \\   
	&  \left.  \textrm{\quad\ \ and } \phi_{[d_{i}, d_{i+1}]} \in C_{seg}([0, d_{i+1} - d_{i}]) \textrm{ for all }0 \le i < m \right\}.
\end{align*}
Here, instead of spaces $C^{\le 1}_{pol}$ and $C^{+}_{r,comp}$ from 
\cite{BKV}, we will use 
$C^\infty([0, \delta], [0,1])$ and $C^{+}_{r,sol}$, respectively. 
It is clear that 
$C^{\le 1}_{pol} \subset C^\infty([0, \delta], [0,1])$ and a similar
fact to Lemma 4.3 from \cite{BKV} can be
stated:
\begin{lemma}
	\label{lem:comp-spaces}
	Assume $g$ in \eqref{eqn:Eg} is of class $C^\infty$ with $g(\xi) \ge 0$ for $\xi \ge 0$. 
	Let $\phi \in C^{+}_{r,sol}$
	and $l \in (-1, 0]$ be such that $\psi_{[-1, l]} \in C^{+}_{seg}$. Then, 
	$x^{\phi}_{[0, l+1]} \in C^\infty([0, l+1], \R^+)$.
\end{lemma}
\begin{proof} 
	Either (a) $\phi_{[-1, l]} \in C^{>1}$
	and, in that case, $x^{\phi}(s) = x(0)   e^{-c   s}$ for $s \in [0, l + 1]$,
	or (b) $\phi_{[-1, l]} \in C^{\le 1}$ and then $x^{\phi}|_{[0, l+1]}$ 
	is a solution to a smooth ($C^\infty$) 
	non-autonomous ODE $x'(t) = -c   x(t) + d   h(t)$, with $h(t) = (g \circ \phi)(t) > 0$.
	It has a solution 
	\begin{equation}
		x^{\phi}(t) = e^{-ct}\phi(0) + d\int_0^t e^{s-t}h(s)ds \ge 0 \qquad t \in [0, l+1] \label{eq:explicit-solution}
	\end{equation}
\end{proof}

The two key features of Lemma 4.3 in \cite{BKV} 
was that (a) it provided explicit formulas for $x^{\phi}_{[0, l+1]}$ 
and (b) it guaranteed that $x^{\phi}_{[0, l+1]}$ had finite number of intersections
with the line $x = 1$. This guaranteed that $\Gamma(1, \phi) \in C^{+}_{r, comp}$
for all $\phi \in C^{+}_{r, comp}$. It was possible because $h(s)$
in \eqref{eq:explicit-solution} was amenable to almost symbolic computations
and the solutions were in the form of an exponential-polynomial products.
In the current work we need to circumvent two problems: first, of an arbitrarily complicated
$h(s)$ caused by a general representation of $C^\infty$ solution segment $\phi$ 
and second, we cannot prove \emph{a priori} $\Gamma(1, \phi) \in C^{+}_{r, sol}$,
as we do not explicitly know the number of intersections. Instead,
our algorithm will prove along the way that $x_{[k, k+1]} \in C^{+}_{r, sol}$
for $k \in \{0, 1, \ldots, K \}$ for a given $x_{[-1, 0]} \in C^{+}_{r, sol}$.
To do this, we need to introduce basic notions of the rigorous 
integration procedure from \cite{SZ2}. This algorithm 
works with the piecewise Taylor representation of solution segments. 
Let $\phi$ be some $C^{n+1}$ function,  
$\phi : [t, t + \delta) \to \R$. \emph{The forward jet of $\phi$ at $t$} is: 
\begin{equation*}
	\left(J^{[n]}_{t}\phi\right)(s) = \sum_{k=0}^{n} \phi^{[k]}(t)   s^k \quad s \in [t, t+s), 
\end{equation*}
where $\phi^{[k]}(t) = \frac{\phi^{(k)}(t)}{k!}$
are coefficients of Taylor expansion of $\phi$. 
We identify the jet with a vector in $\R^{n+1}$ in a standard way: 
$J^{[n]}_{t}\phi \equiv \left(\phi^{[0]}(t), \ldots,\phi^{[n]}(t)\right) \in \R^{n+1}$. 
This allows to write $J^{[n]}_{t} \phi \in A \subset \R^{n+1}$, and
we will use that frequently. 
\begin{definition} 
	The forward Taylor
	representation of $\phi : [t, t+\delta)$ is the pair 
	$$(J^{[n]}_{t}\phi, \xi) \subset \R^{n+1} \times C^{0}([t, t+\delta], \R)$$
	such that the Taylor formula is valid for $g$:
	\begin{equation*}
		\label{eq:forward-taylor}
		\phi(s) = \left(J^{[n]}_{t}\phi\right)(s) + (n+1) \int_t^\delta \xi(s)   (t-s)^n ds,\quad s \in [t, t + \delta)
	\end{equation*}   
\end{definition}
Obviously, $\xi(s) = \phi^{[n+1]}(s)$, and that we are assuming $\phi^{[n+1]}(s)$ bounded
over $[t, t+\delta)$. In what follows, we will identify any such $\phi$
with its \emph{representation} $(J^{[n]}_{t}\phi, \xi)$, and we
will write $\phi \equiv (J^{[n]}_{t}\phi, \xi)$. For a given vector $j \in \R^{n+1}$ 
and a function $\xi$ we identify $(j, \xi)$ with a function 
$\phi : [0, \delta) \to \R$ given by \eqref{eq:forward-taylor}
with $J^{[n]}_0 \phi = j$. The order $n$ is known from the context 
(the length of the vector $j$). We will write $(j, \xi)(t)$
to denote the evaluation $\phi(t)$.

\begin{definition}
	A grid of size $h_p = \frac{1}{p}$
	over the base interval $[-1, 0]$ is the set of points $\left(t_i\right)_{i=0}^{p}$
	with $t_i = -i   h$. We have $t_p = -1$ and $t_0 = 0$.
\end{definition}
We will usually drop subscript $p$. In what follows the constant $h$
will always mean the grid step size $h_p$ and the variable $\epsi$ 
will be such that $0 \le \epsi < h$.

\begin{definition}
	\label{def:cetap}
	For a vector $\eta = (n_1, \ldots, n_p)$ 
	the space $C^\eta_p$ is defined as:
	\begin{eqnarray*}
		C^\eta_p & = & \left\{ x : [-1, 0] \to \R: \exists j_i \in \R^{n_i + 1},  \exists \xi_i \in C^{0}([0, h], \R) \right. \\ 
		& & \left. \textrm{ such that }  x_{[t_i, t_{i-1})} \equiv (j_i, \xi_i)\right\}
	\end{eqnarray*}
\end{definition}
Please note that functions in $C^\eta_p$ can be discontinuous
at the grid points $t_i$. However, if the function $x$ is 
$C^{n+1}$-smooth on the whole $[-1, 0]$, then obviously $x \in C^n_p$. 
Later, the discontinuities at the grid points will allow us to 
handle discontinuities arising from solving \eqref{eqn:Eg}.

\begin{definition}
	A ($p$, $\eta$)-representation (or just the representation)
	of a function $x \in C^\eta_p$ is the collection 
	of $(z, j_1, \ldots, j_p, \xi_1, \ldots, \xi_p)$,
	where $z = x(0)$ and $x_{[t_i, t_{i-1}]} \equiv (j_i, \xi_i)$
	as in Definition~\ref{def:cetap}.
\end{definition}
Please note, that $(z, j_1, \ldots, j_p, \xi_1, \ldots, \xi_p) = (z, j, \xi)$
can be think of as an element of $\R^M \times (C^0([0, h], \R))^p$
with additional structure (coefficients of 
$j$ divided among $p$-grid and $\eta$-orders).
Here $M = M(p, \eta) = 1 + \sum_{j=1}^{p} (\eta_i + 1)$.
In the computations, the structure $\eta$ can be deduced
from the structure of components of $j$.
We will again just write that $x = (z, j, \xi)$
and we might write $(z, j, \xi)(t)$ to denote
the evaluation of the corresponding $x$ at $t$.
Note that 
if $t = -i   h + \epsi$, then $(z, j, \xi)(t) = (j_i, \xi_i)(\epsi)$
and if $t = 0$, then $(z, j, \xi)(t) = z$.

Now we need finite description of the
subsets of $C^\eta_p$ to put them into computer.
Let $Z, Y$ are closed intervals in $\R$ and $\xi \in C^0(Z, \R)$.
By writing $\xi \in Y$ we will denote the fact that 
$\left[\inf_{x \in Z} \xi(x), \sup_{x\in Z} \xi(x)\right] \subset Y$. 
\begin{definition}
	\emph{The representable ($p$,$\eta$)-functions-set} (or just \emph{f-set})
	is a pair $$(A, \Xi) \subset \R^M \times \IR^p.$$ The support 
	of the f-set is $X(A, \Xi) \subset C^\eta_p$: 
	\begin{equation}
		x \in X(A, \Xi) \iff \exists (z, j) \in A, \exists \xi \in \Xi: x = (z, j, \xi).
	\end{equation}
\end{definition}
We will often identify the f-set with its support. We will
use projections: $z(A)$, $j(A)$, $j_i(A)$, to denote
the respective components of the representation. 
For $x \in C^\eta_p$
by $X(x)$ we will denote the smallest f-set
that contains the representation $(z, j, \xi)$ of $x$.
If we want to stress out what ($p$, $\eta$)-representation
we are using (e.g. when the space $C^\eta_p$ is not clear from the context), 
we will write $X^\eta_p(x)$. We use the convention that for $n \in \N$
$X^n_p = X^\eta_p$, with $\eta = (n, \ldots, n)$ - a uniform
order over all grid points. We also note that it is easy,
given an ($p$, $\eta$)-set $X$ with $\eta_i > 0$ for all $i$, 
to obtain a ($p$, $\lambda$)-set $Y$ representing 
$\frac{dg}{dt} $ for all $g \in X$ (with $\lambda_i = \eta_i - 1$).
We will therefore write $X'$ to denote this set. 
It will be used in Algorithm~\ref{alg:inewton}.

The work \cite{SZ2} provides two (families of) algorithms to propagate forward in time
solutions to a general DDE of the form 
\begin{equation}
	\label{eqn:smooth-dde}
	x'(t) = F\left(x(t), x(t-h), x(t-2h), \ldots, x(t-\tau)\right),
\end{equation}
where $F$ is smooth enough.
The algorithms are denoted by $\mathcal{I} : C^\eta_p \to C^\lambda_p$
and $\mathcal{I}^\epsi : C^\eta_p \to C^\zeta_p$, respectively, such that they
guarantee for the semiflow $\varphi$ associated to DDE~\eqref{eqn:smooth-dde} 
the following:
\begin{itemize}
	\item[(i)] $\varphi(h, x) \in \mathcal{I}(X(A, \Xi)) = X(B, \Lambda) \subset C^\lambda_p$, for all $x \in X(A, \Xi)$;
	\item[(ii)] $\varphi(\epsi, x) \in \mathcal{I}^\epsi(X(A, \Xi)) = X(B, Z) \subset C^\zeta_p$ for all $x \in X(A, \Xi)$,
	such that $\varphi(\epsi, x) \in C^\zeta_p$.
	\item[(iii)] if $x \in X(A, \Xi) \cap C^\infty([-1, 0], \R)$, then it is guaranteed  that 
	$J^{[\zeta_{i}]}_{-i  h} \varphi(\epsi, x) \in j_i(I_{\epsi}(X(A, \Xi)))$, for
	all $i \in 1, \ldots, p$.
\end{itemize}
The property (i) guarantees that the solution to any initial data in
a ($p$,$\eta$)-fset are representable after full time step $h$. The properties
(ii) and (iii) guarantees that, for solutions smooth enough, the shift by a
step $\epsi$ smaller than the full step $h$ is well defined, and the shifted
segment can be described by some ($p$,$\zeta$)-representation.
The grid size $p$ of the representations must be the same, but 
the structure of the result sets given by $\lambda$ and $\zeta$ can be 
different from input $\eta$ for technical reasons, important to obtain
rigorous results of good numerical quality. The details are given in
\cite{SZ2}, but Reader can think of
$\mathcal{I}$ and $\mathcal{I}^\epsi$ just as functions 
$C^\eta_p \to C^\eta_p$, with a fixed $\eta$, 
for the sake of presentation of the methods.
It is also easy to see that the iteration $\mathcal{I}^p$ realizes
the classical method of steps for the DDEs, and a Poincar\'e
map $P : S \supset D \to S$ might be realized as 
$\mathcal{I}^{[\epsi_1, \epsi_2]} \circ \mathcal{I}^m$,
if the return time function $t_P : S \supset D \to \R^+$ can be 
proved to be continuous and estimated by 
$t_P(D) \subset m   h + [\epsi_1, \epsi_2]$,
$0 \le \epsi_1 \le \epsi_2 < h$. The method from
\cite{SZ2} can be used to prove that the Poincar\'e 
maps are well defined in that matter. 

\begin{definition}
	Assume $p \in \N_+$ is the grid size. 
	
	\emph{The representation} of a function $\phi \in C^{+}_{r, sol}$
	is a collection of objects:
	\begin{equation*}
		\left(z, \left((d_i, z^i, j^i, \xi^i)\right)_{i=0}^{m-1} \right)
	\end{equation*} 
	such that $(z^i, j^i, \xi^i) \in C^{\eta^i}_{p}$, 
	$\phi_{[d_i, d_{i+1}]} = (z^i, j^i, \xi^i)_{[-1, -1 + d_{i+1} - {d_i}]}$, and
	$z = \phi(0)$.
\end{definition}
In other words, each ($p$,$\eta^i$)-representation $(z^i, j^i, \xi^i)$ is
given over whole basic interval, but we assume it is valid only
on the initial segment of length $\delta_i = d_{i+1} - {d_i}$ (it might span 
several grid points, and the end point is not necessarily at the grid point,
i.e. usually $\delta_i = k_i   h + \epsi_i$, $k_i \in \N$, $\epsi_i \in [0, h)$).
Each jet $j^i_p$ representing the function at $t = -p  h = -1$
is matched with the jet of $\phi$ at time $d_{i}$, 
See Figure~\ref{fig:representation}. The distinguished
element $z$ is called \emph{the head} of the solution
(as in this analogue: the solution $x_t^\phi$ is a snake moving in
the phase space, and the value $z$ represents 
$x_t^{\phi}(0)$ - its head - that is moving according to \eqref{eqn:Eg}).
Analogously, we define \emph{f-set in $C^{+}_{r, sol}$} as
a collection of f-sets: 
\begin{equation*}
	\left(Z, \left((d_i, A^i, \Xi^i)\right)_{i=0}^{m-1} \right).
\end{equation*}
The smallest f-set containing the representation of $\phi \in C^{+}_{r, sol}$
is denoted again as $X(\phi)$.

\begin{figure}
	\begin{center}
		\includegraphics[width=9cm]{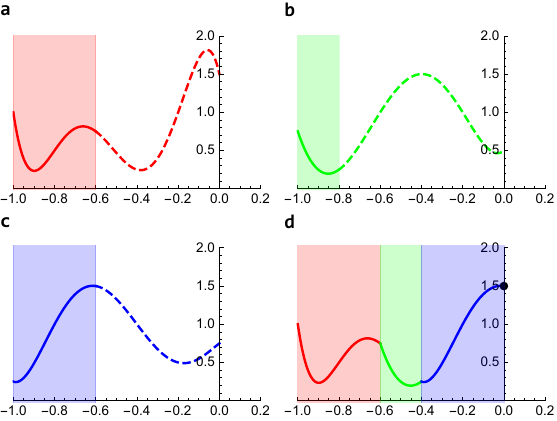}
	\end{center}
	\caption{\label{fig:representation}Schematic picture showing the 
		representation of a function in $C^{+}_{r, sol}$ (d), build
		with $C^n_p$ representations (a-c). The dashed lines show
		parts of the $C^n_p$ representations not used in computations, while
		the coloured backgrounds show ranges of respective lengths $\delta_i = d_i - d_{i+1}$.
		In this case $i \in \{0, 1, 2\}$. In (d), the ranges are shown
		in respective colours between points $d_i$ and $d_{i+1}$, where the 
		$C^n_p$ representation holds. The black dot
		in (d) is $z$ - the head of the representation.}
\end{figure}

\begin{rem}[A technical note about the representation $\phi \in C^+_{r, sol}$]
	It might be strange mathematically to represent $\phi$ by ,,gluing''
	together representations of $C^{\eta^i}_p$, especially, when there is 
	a lot of unnecessary data, as each piece is valid only on (possibly) small
	interval $[-1, -1+\delta_i)$. However, with such a representation
	we get a lot easier implementation of the algorithm to propagate
	data using the algorithms $\mathcal{I}$ and $\mathcal{I}^\epsi$,
	as the data is already in a good format.
\end{rem}

We will now present Algorithm~\ref{alg:main} to 
compute rigorous estimates on $x^\phi$ for $\phi = e^{-a(s+1)}$ 
over a long interval $[-1, M]$, $M \in \N$.  
The algorithm successively computes f-sets representing the solution
as an element of $C^+_{r,sol}$ over consecutive intervals
$[k-1, k]$, $k \in \N$. In what follows, we denote
by $\mathcal{I}_{\le 1}$, $\mathcal{I}^\epsi_{\le 1}$, 
$\mathcal{I}_{> 1}$, $\mathcal{I}^\epsi_{> 1}$ the 
procedures from \cite{SZ2} that
are configured to compute estimates on semiflows for smooth 
systems $x'(t) = -a x(t) + b   f(x(t-1))$ and $x'(t) = -a x(t)$, 
respectively. 
To organize the presentation, we first give a 
very general description of Algorithm~\ref{alg:main},  
then we shortly describe what the three subroutines
used in it do, finally we state 
the main Lemma~\ref{lem:alg-main} about the correctness
of the procedure. As the subroutines are very technical, 
their details are postponed to 
Appendix~\ref{secA1}, so is the proof of Lemma~\ref{lem:alg-main}.
\begin{algorithm}[h]
	\caption{Main algorithm\label{alg:main}}
	\begin{algorithmic}[1]
		\Require 
		$c, d$ - parameters of \eqref{eqn:Eg} and
		the f-set $X_0 = X(\phi) \subset C^+_{r,sol}$
		\Ensure 
		An f-set $X_1 \subset C^{+}_{r, sol}$ such that it contains the solution  
		$\Gamma(1, \phi)$ in the sense of Lemma~\ref{lem:alg-main}.
		
		\vspace{0.5em} \hrule \vspace{0.5em}
		\State $U \gets \texttt{shift}(X_0)$ \Comment{shift all subsegments by full time lag}
		\State $V \gets \texttt{recut}(U)$ \Comment{then, cut new segment at points crossing $x = 1$}
		\State $X_1 \gets \texttt{rake}(V)$ \Comment{finally, merge consecutive segments lying whole in $C^{>1}$.}
		\State \Return $X_1$
	\end{algorithmic}
\end{algorithm}
The additional procedures used in Algorithm~\ref{alg:main} do the following:
\begin{itemize}
	\item Algorithm~\ref{alg:inewton} (Interval Newton's Method) 
	is a standard procedure to prove that there is a single  
	zero of a function in a given neighbourhood. It can also
	be used to narrow the set to be as small as possible.   
	We will use it as a sub-procedure in 
	Algorithm~\ref{alg:recut} but we provide the pseudo-code
	for completeness and to demonstrate places where Algorithm~\ref{alg:recut}
	might fail later. What is more, the successful execution of the algorithm will guarantee 
	that the Property~(P2) is satisfied. 
	
	\item Algorithm~\ref{alg:shift} - ,,shift'': realizes one iteration of the method of steps
	to convert $\phi$ into $\Gamma(1, \phi)$. The resulting description (pre-representation)
	is not a valid $C^+_{r, sol}$ f-set representation and it needs to be corrected for
	the subsequent iterations. 
	
	\item Algorithm~\ref{alg:recut} - ,,re-cut'': makes sure to find
	all intersections of the pre-representation of $\Gamma(1, \phi)$ 
	with the line $x = 1$ and convert the pre-representation into 
	representation in $C^+_{r, sol}$.
	
	\item Finally, Algorithm~\ref{alg:rake} - ,,rake'': removes all consecutive
	segments of the solution that lie in $C^{>1}$ and combines them
	into one segment. This reduces the size of the representation
	and in turn lessens the computational complexity of the subsequent
	steps. It also increases the accuracy, as the $\mathcal{I}_{>1}$ could
	be used for longer uninterrupted periods in Algorithm~\ref{alg:shift}
	resulting in less \emph{wrapping-effect} impact of Line~\ref{line:large-wrapping} 
	in Algorithm~\ref{alg:shift}.
\end{itemize}

\begin{lemma}
	\label{lem:alg-main}
	Let $\phi \in C^+_{r,sol}$.
	If succeeds, Algorithm~\ref{alg:main} applied to the f-set representation
	$X(\phi)$ produces an f-set 
	$X = \left(Z, \left((\bar{d}_i, A^i, \Xi^i)\right)_{i=0}^{m-1} \right) \subset C^+_{r,sol}$
	with $\bar{d}_i \in \IR$ such that 
	\begin{enumerate}[label={\Alph*})]
		\item  \label{alg-main-1} 
		if $d_i$ are the points where $\Gamma(1, \phi)(d_i) = 1$, 
		then $d_i \in \bar{d}_i$,
		\item  \label{alg-main-2}
		if $\Gamma(1, \phi)_{[d_i, d_{i+1}]} \in C^{\le 1}$, 
		then $\Gamma(1, \phi)_{[d_i, d_{i+1}]} \in X_{[d_i, d_{i+1}]}$.
	\end{enumerate}
\end{lemma}

Now we can state the results about the existence of
periodic solutions to Eq.~\ref{eqn:Eg} for 
concrete values of parameters. Before we do that however,
we need to discuss shortly the floating-point arithmetics
and its presentation in the article, following the convention used in \cite{SZ2}:
\begin{rem}[On the representation of floating-point numbers from computations]
	\label{rem:floating-points}
	Due to the very nature of
	the implementation of real numbers in computers, numbers like
	$0.1$ are \emph{not representable} \cite{ieee-754}, i.e. cannot be stored
	in the memory exactly. On the other hand, many representable numbers could
	not be shown in the text in a reasonable, human-readable way.
	As we use the \emph{interval arithmetic} to produce
	rigorous estimates on the true values, 
	whenever we present \emph{any output of the
		computer program} $x$ as a decimal number with non-zero fraction part, 
	then we have in mind that this in fact 
	represents some other number $y$ such that $y \in B(x, \epsilon^M)$,
	where $\epsilon^M$ is the machine precision of the double-precision 
	floating point numbers as defined in \cite{ieee-754}. 
	This convention applies also to intervals: if we write an interval $[a_1, a_2]$,
	then we understand this as $[b_1, b_2]$ with $b_1$, $b_2$ - representable and such that $[a_1, a_2] \subset [b_1, b_2]$.
	Finally, if we write a number in the following manner: 
	$d_1 . d_2  s d_k {}^{u_1 u_2  s u_m}_{l_1 l_2  s l_m}$ 
	with digits $l_i, u_i, d_i \in  \{0,..,9\}$ then it represents
	the following interval
	\begin{equation*}
		\left[ d_1 . d_2  s d_k l_1 l_2  s l_m, d_1 . d_2  s d_k u_1 u_2  s u_m \right].
	\end{equation*}
	For example $12.3_{456}^{789}$ represents the interval $[12.3456, 12.3789]$
	(here we also understand the numbers taking into account the former conventions).
\end{rem}

Now we present the result proven with the computer assistance:
\begin{theorem}
	\label{thm:main-comp-assisted}
	For $k$, $c$, $d$, $T$, $\omega_p$ as in Table~\ref{tab:results},
	the system \eqref{eqn:Eg} with $g(\xi) = \xi^k$
	has a periodic orbit with the basic period $\omega_p$ satisfying (P).
\end{theorem}
\begin{proof}
	Let $\phi(s) = e^{-c(s+1)}$. 
	We set $X_0 = X(\phi) = X\left(\phi(0), \left(-1, X^{n}_{p}\left(\phi\right)\right)\right)$.
	Then we successively compute
	\begin{equation}
		\label{eq:computer-assisted}
		X_k = \left(Z, \left(\bar{d}^k_i, A^{k,i}, \Xi^{k,i}\right)_{i=0}^{m_k-1}\right) := \mathtt{main}\left(X_{k-1}\right), \quad k \in 1, \ldots, T.
	\end{equation}
	Note, the sets $\bar{d}^k_i \in \IR$ contains the true $d^k_i$'s, according to Lemma~\ref{lem:alg-main}.
	Next, we show that there exists $\bar{d}^{T-1}_{m_{T-1}}$
	and $\bar{d}^{T}_1$, such that:
	\begin{enumerate}[label={\Alph*)}]
		\item $\omega_p \in T + 1 + \bar{d}^{T}_1 \subset [T, T+1]$;
		\item $1 = x^\phi(\omega_p - 1) \in (X_T)(\bar{d}^{T}_1)$;
		\item $\bar{d}^{T-1}_{m_{T-1}} - \bar{d}^{T}_1 \ge 0$; \label{cond:lenght}
		\item $X(A^{T-1,m_{T-1}},\Xi^{T-1,m_{T-1}}) \in C^{\ge 1}$ and 
		$X(A^{T,0},\Xi^{T,0}) \in C^{\ge 1}$.
	\end{enumerate}
	Note, that all those conditions imply that $\Gamma(\omega_p-1, \phi) \in C^{\ge 1}$,
	and therefore $\Gamma(\omega_p, \phi) = \phi$. Especially,
	$L = \bar{d}^{T-1}_{m_{T-1}} - \bar{d}^{T}_1 + 1$ is the length
	of the solution segment $x^\phi_{[\omega_p-1 - L, \omega_p - 1]} \ge 1$,
	so the condition \ref{cond:lenght} imply $L > 1$ and guarantees 
	$\Gamma(\omega_p, \phi) = \phi$.  
\end{proof}

Parts of the proof are computer assisted: propagation of \eqref{eq:computer-assisted}
and checking the four assumptions A-D. 
Programs to redo computations are available in \cite{RSwww}.
Some of the orbits proved in Theorem~\ref{thm:main-comp-assisted} are 
shown in Figures~\ref{fig:plots-k2}~and~\ref{fig:plots-khalf}, in particular
we point out the last plot in Figure~\ref{fig:plots-k2}, as it shows
a complicated periodic orbit $p$ for the values of parameters $k=2$, $c = 5\pi/(3\sqrt{3})$, $d=7.95$
for which a Hopf bifurcation occurs near the equilibrium point $\frac{c}{d}$. In a 
future paper it will be shown that there is a 
heterocilinic connection to $p$ from a 
small amplitude periodic solution 

\begin{table}[h]
	\setlength{\tabcolsep}{10pt}
	\renewcommand{\arraystretch}{1.5}
	\caption{Values of parameters $k$, $c$, $d$ in \eqref{eqn:Eg}, with $g|_{[0,1)}(\xi) = \xi^k$
		and the estimated period $\omega_p$ of 
		the periodic solutions. The parameters $T$, $n$, $p$ are as in 
		the proof of Theorem~\ref{thm:main-comp-assisted}, together
		with the estimated length $L$ of the last interval, over which the 
		solution $x^\phi$ lies above line $x = 1$. Presented values, when
		not computer-representable, are stored during the computations
		as intervals containing them as discussed in Remark~\ref{rem:floating-points}.
		In particular, the value $c = 5\cdot\pi/(3\cdot\sqrt{3})$ means that in computations
		the rigorous interval containing this value was used, and the proof is valid for
		all values of $c$ in that interval.}	\label{tab:results}
	\begin{tabular}{c|c|c|c|c|c|c|c}
		\toprule
		$k$ & $c$ & $d$ & $T$ & $p$ & $n$ & $\omega_p$ & $L$ \\
		\midrule
		$2$ & $2$ & $20$ & $3$ & $128$ & $4$ & $2.2475762_{8868803}^{9130349}$ & $1.19642835_{388474}^{650057}$ \\			
		$2$ & $2$ & $5$ & $10$ & $128$ & $4$ & $9.718074_{41624915}^{68592668}$ & $1.173666_{27746896}^{65610938}$ \\
		$2$ & $5$ & $13.5$ & $11$ & $128$ & $4$ & $10.3996_{6084251935}^{8118185736}$ & $1.051_{39513368493}^{42163578326}$ \\
		$2$ & $4$ & $12.71$ & $16$ & $512$ & $5$ & $15.50313_{451849748}^{939964593}$ & $1.1477_{3973007044}^{4713794320}$ \\			
		$2$ & $\frac{5\cdot\pi}{3\cdot\sqrt{3}}$ & $7.95$ & $15$ & $256$ & $4$ & $14.70203_{302733336}^{563332337}$ & $1.10426_{444733693}^{787605344}$ \\
		$\frac{1}{2}$ & $2.5$ & $5.1$ & $9$ & $128$ & $4$ & $8.714_{57545370406}^{62830383538}$ & $1.154_{25881778764}^{36842629369}$	\\		
		$\frac{1}{2}$ & $2.5$ & $5.2$ & $12$ & $256$ & $5$ & $11.4322_{6910337948}^{7232415869}$ & $1.15709_{158785186}^{779182356}$ \\			
		$\frac{1}{2}$ & $3.8$ & $5.205$ & $17$ & $1024$ & $5$ & $15.6237_{3745723471}^{4300377583}$ & $1.0616_{4001088705}^{5181449065}$ \\
		$\frac{1}{2}$ & $3.8$ & $5.2$ & $15$ & $1024$ & $5$ & $13.741833_{36789089}^{66133022}$ & $1.05535_{361584126}^{410986789}$ \\
		$\frac{1}{2}$ & $3.8$ & $5.195$ & $15$ & $1024$ & $5$ & $13.63633_{685525799}^{711062829}$ & $1.054959_{16531613}^{80790828}$ \\
		$\frac{1}{2}$ & $3.81$ & $5.215$ & $15$ & $1024$ & $5$ & $13.743261_{57465981}^{86363402}$ & $1.05439_{789658996}^{837966255}$ \\			
		$\frac{1}{2}$ & $4.65$ & $6.51$ & $16$ & $1024$ & $5$ & $14.735_{51192774460}^{70287262517}$ & $1.041_{71346771675}^{98571950225}$ \\			
		$\frac{1}{2}$ & $4.66$ & $6.53$ & $15$ & $1024$ & $5$ & $14.730_{16163977130}^{21436894118}$ & $1.047_{84259656712}^{96838456861}$ 	
	\end{tabular}
	
\end{table}

\begin{figure}
	\begin{center}
		\includegraphics[height=4.5cm]{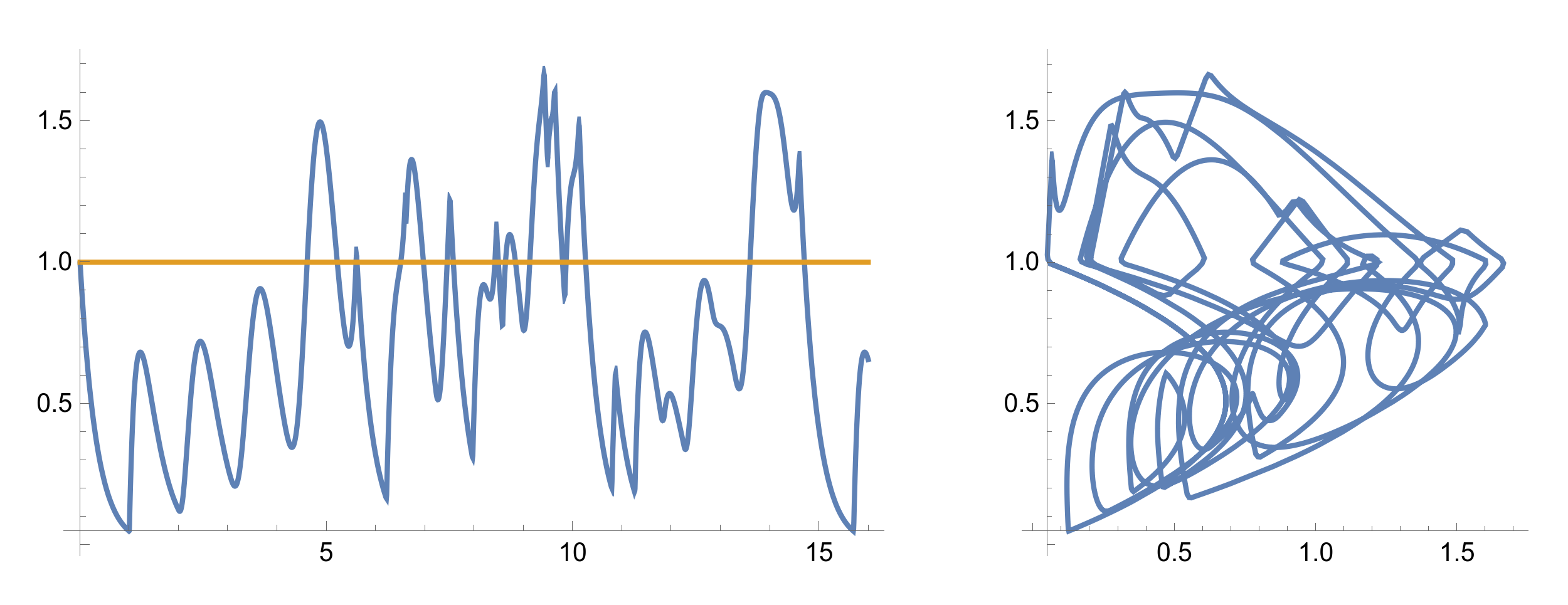}
		
		\includegraphics[height=4.5cm]{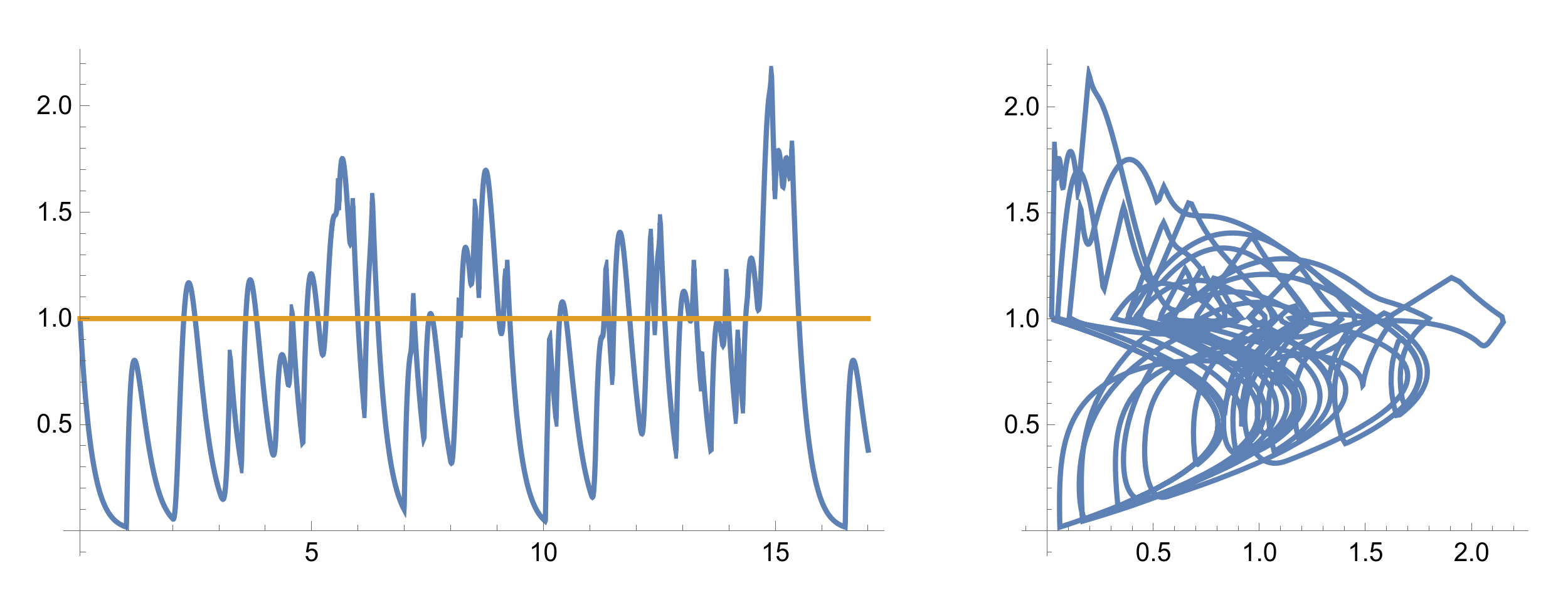}
	\end{center}
	\caption{\label{fig:plots-k2}Plots of the numerical representations of solutions
		corresponding to entries in Table~\ref{tab:results} for $k=2$, presented in the same fasion as in Figure~\ref{fig:figk2}.
		Solutions correspond to small intervals containing values of parameters: 
		$c=\frac{5\cdot\pi}{3\cdot\sqrt{3}}$ (i.e. parameter of Hopf biffurcation at $\xi_0$), $d=7.95$ (up), and $c=4$, $d=12.71$ (bottom).}
\end{figure}

\begin{figure}
	\begin{center}
		\includegraphics[height=4.5cm]{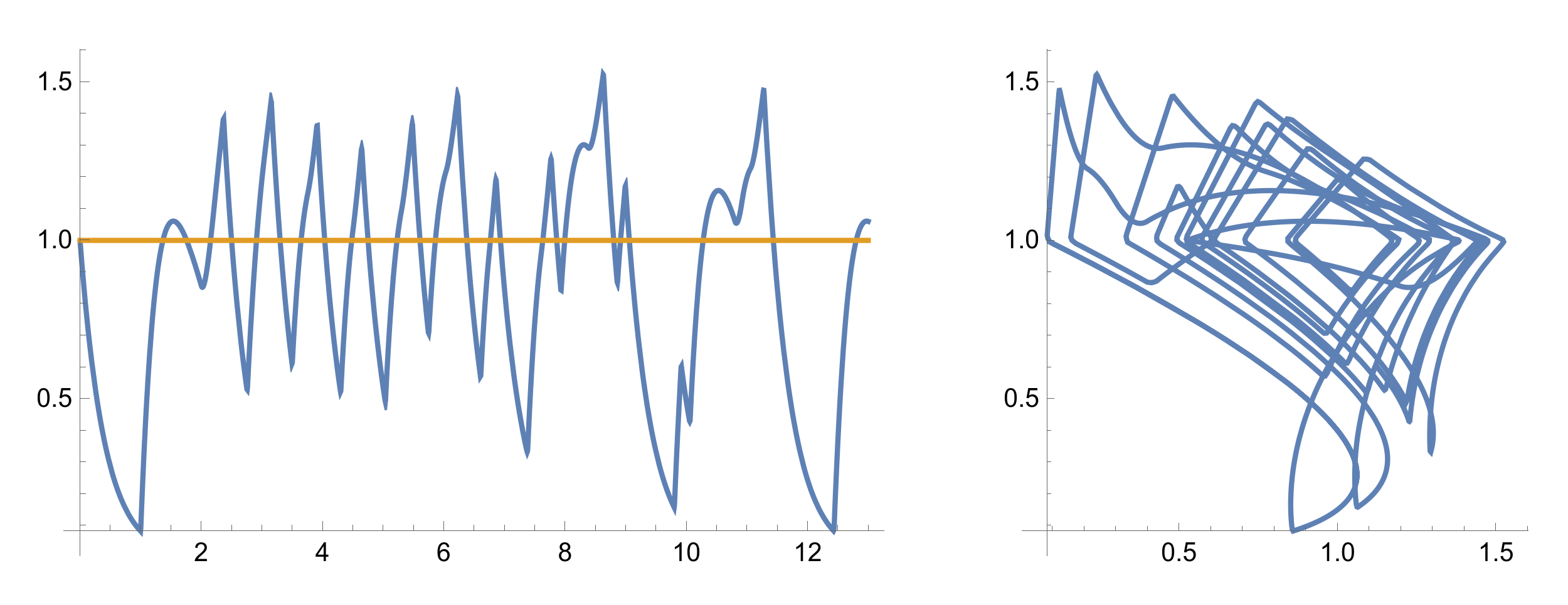}
		
		\includegraphics[height=4.5cm]{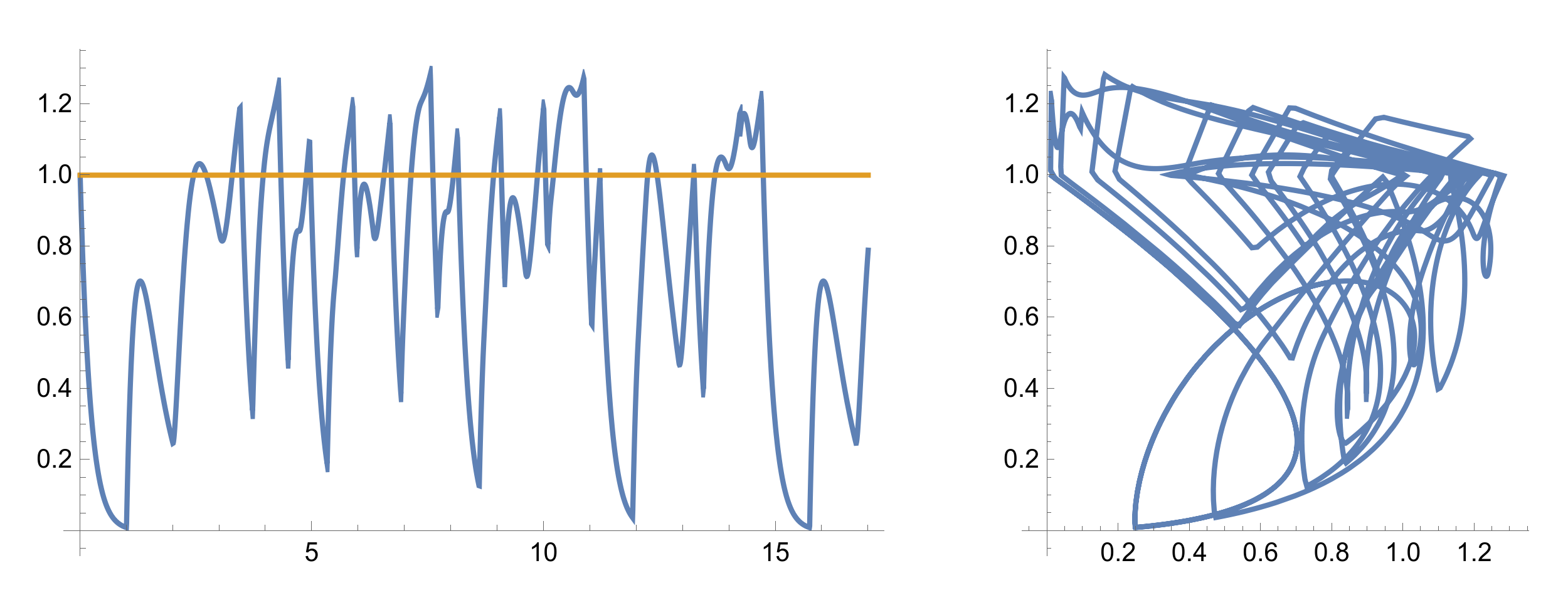}		
	\end{center}
	\caption{\label{fig:plots-khalf}Plots of the numerical representations of solutions
		corresponding to entries in Table~\ref{tab:results} for $k=1/2$, presented in the same fasion as in Figure~\ref{fig:figk2}.
		Solutions correspond to small intervals containing values of parameters: 
		$c=2.5$, $d=5.2$ (up), and $c=4.65$, $d=6.51$ (bottom).}
\end{figure}



\clearpage
\begin{appendices}
	
	\section{Details on the algorithms}\label{secA1}
	
	\begin{algorithm}[h]
		\caption{\texttt{inewton}: for a (forward) Taylor representation $(J, \Xi)$ of $f$ over $I = [0, \delta]$
			proves either: 1) there is no root of $f$ in $I$, 2) there is exactly one root $t_0 \in [0, \delta]$, or
			c) raises an error if multiple zeroes are possible. We use notation $[x]$ to denote a set $[x] \subset \IR$ 
			such that it contains $x \in [x]$. Operations $\diam([a, b]) := b - a$ and $\midpoint([a, b]) := \frac{a+b}{2}$ 
			denote the diameter and the middle point of the set, respectively. \label{alg:inewton}}
		
		\begin{algorithmic}[1]
			\Require $(J, \Xi) \subset \R^{n+1} \times \IR$, $I = [0, \delta]$, $\epsilon_M \in \R_+$, $M \in \N$.
			\Ensure $Z_x = \{ t \in I : x(t) = a \} = \{ t_0 \} \subset \bar{t}_0 \subset [0, \delta]$ or $Z_x = \emptyset$
			for all $x = (j, \xi) \in (J, \Xi)$. It returns $\bar{t}_0 \in \IR$.
			
			\vspace{0.5em} \hrule \vspace{0.5em}
			
			\State $(J', \Xi') \gets \left((i   J_i)_{i=1}^{n}, (n+1)  \Xi\right)$ \Comment{representation of the derivative}
			\State $[t_0] = [0, \delta]; \quad t_0 = \frac{\delta}{2}$ \Comment{choose $t_0 \in \bar{t}_0$}
			\State $k = 0$
			\While{$k < M \land \diam(t_k) > \epsilon_M$} \Comment{common stop conditions}
			\If{$[t_k] = \emptyset \lor a \notin (J, \Xi)([t_k])$} \Comment{there is no chance for solution in $I$}
			\State \Return $\emptyset$
			\EndIf
			\If{$0 \in (J, \Xi)'([t_k])$} \Comment{Newton's operator ill-defined}
			\State Raise an Error and Abort
			\EndIf	
			\State $[t_{k+1}] = \left(t_k - \frac{(J, \Xi)(t_k) - a}{(J, \Xi)'([t_k])}\right) \cap [t_k] \quad t_{k+1} = \midpoint([t_{k+1}])$
			\State $k \gets k + 1$
			\EndWhile
			\If{$t_k \subset \interior(t_0)$}  \Comment{one need strong inclusion to prove uniqueness}
			\State \Return $t_k$
			\Else
			\State Raise an Error and Abort
			\EndIf
		\end{algorithmic}
	\end{algorithm}
	
	\begin{algorithm}[h]
		\caption{\texttt{shift}: computes a segment after one full delay\label{alg:shift}}
		\begin{algorithmic}[1]
			\Require $X(W) \subset C^+_{r,sol}$, $W = \left(Z, (d_i, A^i, \Xi^i)_{i=0}^{m-1}\right)$
			for some $m \in \N_+$.
			\Ensure $\left(Z', \left(d_i, A'^i, \Xi'^i\right)_{i=0}^{m-1}\right)$ such that
			$\Gamma(1, x)(0) \in Z'$ and $\Gamma(1, x)_{[d_i, d_{i+1}]} = y_{[-1, -1 + \delta_i]}$
			for some $y \in X(A'^i, \Xi'^i) \subset C^{\eta'_i}_{p}$, for all $x \in X(W)$.
			
			\vspace{0.5em} \hrule \vspace{0.5em}
			\State $head \gets Z$ \Comment{start at the current value of the solution}
			\For{$i = 0$ to $m-1$ } \Comment{then process all subsequent segments}
			\State $\tilde{A} \gets \left(head, j(A^i)\right)$ \Comment{change segment head to the correct value}
			\If{$(\tilde{A}, \Xi^i)_{[-1, -1 + \delta_i]} \subset C^{>1}$}
			\State $(A'^i, \Xi'^i) = \mathcal{I}_{>1}^p\left(\tilde{A}, \Xi^i\right)$ \Comment{Propagate with the correct part of \eqref{eqn:Eg}}
			\ElsIf{$(\tilde{A}, \Xi^i)_{[-1, -1 + \delta_i]} \subset C^{\le 1}$}
			\State $(A'^i, \Xi'^i) = \mathcal{I}_{\le 1}^p\left(\tilde{A}, \Xi^i\right)$ \Comment{Propagate with the correct part of \eqref{eqn:Eg}}
			\Else
			\State Raise an Error and Abort \Comment{Should not happen for a correct $X(W)$}
			\EndIf
			\State $head \gets \left(A'^i, \Xi'^i\right)(-1 + \delta_i)$ \Comment{update the head to the current value}\label{line:large-wrapping}
			\EndFor
			\State $Z' \gets head$ \Comment{the new head of the whole solution is the last head}
			\If{$1 \in Z'$} \Comment{this would break the subsequent call to \texttt{shift}!}
			\State Raise an Error and Abort
			\EndIf
			\State \Return $\left(Z', (d_i, A'^i, \Xi'^i)_{i=0}^{m-1}\right)$
		\end{algorithmic}
	\end{algorithm}
	
	\begin{algorithm}[h]
		\caption{\texttt{recut}: computes intersections with $x = 1$ and reorganizes the $C^+_{r, sol}$ representation
			accordingly\label{alg:recut}} 
		\begin{algorithmic}[1]
			\Require $W = \left(Z, (d_i, A^i, \Xi^i)_{i=0}^{m-1}\right)$ as in the \underline{output} of
			Algorithm~\ref{alg:shift}.
			\Ensure $W' = \left(Z, (d'_i, A'^i, \Xi'^i)_{i=0}^{m'-1}\right)$ such that
			$W \subset C^+_{r, sol}$ and $X(W) \subset X(W')$, $m' \ge m$.
			
			\vspace{0.5em} \hrule \vspace{0.5em}
			\ForAll{$(d_i, A^i, \Xi^i)$}
			\State $T = $ empty list
			\ForAll {$k \in \N : -k   h \in [-1, -1 + \delta_i]$}
			\State $(\epsi_k, sign) = \mathtt{inewton}\left( (A^i_k, \Xi^i_k), a=1, \epsilon_M=10^{-15}, M=10\right)$
			\If {$\epsi_k \neq \emptyset$} \Comment{$\epsi_k$ is the ,,local time'' within $[-kh, -kh+1]$}
			\State add $(k, \epsi_k)$ to the end of $T$ \Comment{we add global time $kh + \epsi$}
			\EndIf
			\EndFor
			\State $(\tilde{A}, \tilde{\Xi}) \gets (A^i, \Xi^i)$; \quad $\tilde{t} \gets d_i$
			\If {$(\tilde{A}, \tilde{\Xi})$ was computed with $\mathcal{I}^p_{\le 1}$} \Comment{choose the right $\mathcal{I}$ algorithm}
			\State $\left(\tilde{\mathcal{I}}, \tilde{\mathcal{I}}^\epsi\right) = \left(\mathcal{I}_{\le 1}, \mathcal{I}^\epsi_{\le 1}\right)$
			\ElsIf {$(\tilde{A}, \tilde{\Xi})$ was computed with $\mathcal{I}^p_{> 1}$}
			\State $\left(\tilde{\mathcal{I}}, \tilde{\mathcal{I}}^\epsi\right) = \left(\mathcal{I}_{> 1}, \mathcal{I}^\epsi_{>1}\right)$
			\Else \Comment{should not happen}
			\State Raise an Error and Abort
			\EndIf
			\State $(\tilde{A}, \tilde{\Xi}) \gets (A^i, \Xi^i)$; \quad $\tilde{t} \gets d_i$
			\ForAll{$(k, \epsi) \in T$}
			\State Add $(\tilde{t}, \tilde{A}, \tilde{\Xi})$ to the output
			\State $(\tilde{A}, \tilde{\Xi}) \gets \left(\tilde{\mathcal{I}}^\epsi \circ \tilde{\mathcal{I}}_{\le 1}^k\right) (A^i, \Xi^i)$
			\Comment{always start from initial!}
			\State $\tilde{t} \gets d_i + k   h + \epsi$
			\EndFor
			\State Add $(\tilde{t}, \tilde{A}, \tilde{\Xi})$ to the output	
			\EndFor
		\end{algorithmic}
	\end{algorithm}
	
	\begin{algorithm}[h]
		\caption{\texttt{rake}: removes redundant segments in $C^{>1}$\label{alg:rake}}
		\begin{algorithmic}[1]
			\Require $W = \left(Z, (d_i, A^i, \Xi^i)_{i=0}^{m-1}\right) \subset C^+_{r,sol}$ 
			(as in output of 
			\Ensure $W' = \left(Z, (d'_i, A'^i, \Xi'^i)_{i=0}^{m'-1}\right) \subset C^+_{r,sol}$
			such that $F\left(1, X(W)\right) = F\left(1, X(W)\right)$ and $m' \le m$.
			\vspace{0.5em} \hrule \vspace{0.5em}
			\State $i \gets 0$, $m' \gets 0$
			\State $d'_0 \gets d_0$
			\State $d'_1 \gets d_1$
			\While{$i < m-1$}
			\State $(A'^K, \Xi'^K) \gets (A^i, \Xi^i)$
			\State $j \gets i + 1$ 
			\If{$(A^i, \Xi^i)_{[-1, -1 + \delta_i]} \in C^{>1}$} \Comment {if true, check following intervals}
			\While{$j < m-1 \land (A^j, \Xi^j)_{[-1, -1 + \delta_j]} \subset C^{>1}$}
			\State $d'_{m'+1} = d_{j+1}$ \Comment{make the current interval to cover also $[d_j, d_{j+1}]$}
			\State $j \gets j+1$ \Comment{next, try another segment}
			\EndWhile
			\EndIf
			\State $i \gets j$ \Comment{we have skipped all previous intervals}
			\State $m' \gets m' + 1$
			\EndWhile
			\State \Return $\left(Z, \left(d'_i, A'^i, \Xi'^i\right)_{i=0}^{m'-1}\right)$
		\end{algorithmic}
	\end{algorithm}
	
	\begin{proof}[Proof of Lemma~\ref{lem:alg-main}]
		Correctness of Algorithm~\ref{alg:inewton} is well known, see e.g. \cite{KN}
		and references therein, and it produces $\bar{d}_i$ as in \ref{alg-main-1}. 
		Algorithm~\ref{alg:shift} produces rigorous estimates following
		the correctness of the algorithm $\mathcal{I}$ of 
		\cite{SZ2} to propagate any $C^\eta_p$ f-sets.
		Algorithm~\ref{alg:recut} works for all solutions
		such that their segment is smooth enough. Lemma~\ref{lem:comp-spaces}
		ensures that for a segment $\phi_{[d_i, d_{i+1}]} \in C^\infty$
		the solution $F\left(1, \phi_{[d_i, d_{i+1}]}\right)$ 
		(appearing in the output of Algorithm~\ref{alg:shift}
		and as an input to  Algorithm~\ref{alg:recut}) must also be $C^\infty$,
		so applying $\mathcal{I}_\epsi$ is justified. 
		Algorithm~\ref{alg:rake} modifies the solution segment
		only in the sub-segments that lie in $C^{\ge 1}$, thus not 
		altering the ,,important parts'' shown in \ref{alg-main-2}.
	\end{proof}
	
	\begin{rem}[Data stored in the algorithms]
		We use $\delta_i = d_{i+1} - {d_i} > 0$, to denote
		the length of the ,,definition'' intervals in space $C^+_{r, sol}$.
		In the actual implementation \cite{RSwww}, 
		we store $\delta_i$'s instead
		of $d_i$'s, and recompute $d_i$ when needed, using the fact that
		$d_0 = -1$.
	\end{rem}

\end{appendices}

\end{document}